\documentclass[a4paper,twoside]{article}
\usepackage{a4}
\usepackage{amssymb}
\usepackage{amsmath}
\usepackage{upref}
\usepackage[active]{srcltx}
\usepackage[pagebackref,colorlinks,citecolor=blue,linkcolor=blue]{hyperref}
\allowdisplaybreaks[2] 
\newcount\minutes \newcount\hours
\hours=\time
\divide\hours 60
\minutes=\hours
\multiply\minutes -60
\advance\minutes \time
\newcommand{\klockan}{\the\hours:{\ifnum\minutes<10 0\fi}\the\minutes}
\newcommand{\tid}{\today\ \klockan}
\newcommand{\prtid}{\smash{\raise 10mm \hbox{\LaTeX ed \tid}}}
\renewcommand{\prtid}{}
\makeatletter
\def\sectionmark#1{} 
\def\subsectionmark#1{}
\newcommand{\sectnr}{\ifnum \c@secnumdepth >\z@
                 \thesection.\hskip 1em\relax \fi}
\def\@evenhead{\footnotesize\rm\thepage\hfil\leftmark\hfil\llap{\prtid}}
\def\@oddhead{\footnotesize\rm\rlap{\prtid}\hfil\rightmark\hfil\thepage}
\def\tableofcontents{\section*{Contents} 
 \@starttoc{toc}}
\makeatother
\makeatletter
\def\@biblabel#1{#1.}
\makeatother
\makeatletter
\let\Thebibliography=\thebibliography
\renewcommand{\thebibliography}[1]{\def\@mkboth##1##2{}\Thebibliography{#1}
\addcontentsline{toc}{section}{References}
\frenchspacing 
\setlength{\@topsep}{0pt}
\setlength{\itemsep}{0pt}%
\setlength{\parskip}{0pt plus 2pt}%
}
\makeatother
\makeatletter
\def\mdots@{\mathinner.\nonscript\!.%
 \ifx\next,.\else\ifx\next;.\else\ifx\next..\else
 \nonscript\!\mathinner.\fi\fi\fi}
\let\ldots\mdots@
\let\cdots\mdots@
\let\dotso\mdots@
\let\dotsb\mdots@
\let\dotsm\mdots@
\let\dotsc\mdots@
\def\vdots{\vbox{\baselineskip2.8\p@ \lineskiplimit\z@
    \kern6\p@\hbox{.}\hbox{.}\hbox{.}\kern3\p@}}
\def\ddots{\mathinner{\mkern1mu\raise8.6\p@\vbox{\kern7\p@\hbox{.}}%
    \raise5.8\p@\hbox{.}\raise3\p@\hbox{.}\mkern1mu}}
\makeatother
\makeatletter
\let\Enumerate=\enumerate
\renewcommand{\enumerate}{\Enumerate%
\setlength{\@topsep}{0pt}
\setlength{\itemsep}{0pt}%
\setlength{\parskip}{0pt plus 1pt}%
\renewcommand{\theenumi}{\textup{(\alph{enumi})}}%
\renewcommand{\labelenumi}{\theenumi}%
}
\let\endEnumerate=\endenumerate
\renewcommand{\endenumerate}{\endEnumerate\unskip}
\makeatother
\newcommand{\addjustenumeratemargin}[1]{%
\setbox0\hbox{(a)} 
\setbox1\hbox{#1} 
\addtolength{\leftmargini}{-\wd0}%
\addtolength{\leftmargini}{\wd1}%
}
\makeatletter
\def\@seccntformat#1{\csname the#1\endcsname.\quad}
\makeatother
\newcommand{\authortitle}[2]{\author{#1}\title{#2}\markboth{#1}{#2}}
\newcommand{\auth}[2]{{#2. #1}}

\newcommand{\art}[6]{{\sc #1, \rm #2, \it #3\/ \bf #4 \rm (#5), \mbox{#6}.}}
\newcommand{\artnopt}[6]{{\sc #1, \rm #2, \it #3\/ \bf #4 \rm (#5), \mbox{#6}}}
\newcommand{\artprep}[3]{{\sc #1, \rm #2, \rm #3.}}

\newcommand{\book}[3]{{\sc #1, \it #2, \rm #3.}}
\newcommand{\AND}{{\rm and }}
\RequirePackage{amsthm}
\newtheoremstyle{descriptive}%
  {\topsep}   
  {\topsep}   
  {\rmfamily} 
  {}          
  {\bfseries} 
  {.}         
  { }         
  {}          
\newtheoremstyle{propositional}%
  {\topsep}   
  {\topsep}   
  {\itshape}  
  {}          
  {\bfseries} 
  {.}         
  { }         
  {}          
\theoremstyle{propositional}
\newtheorem{thm}{Theorem}[section]
\newtheorem{prop}[thm]{Proposition}
\newtheorem{lem}[thm]{Lemma}
\theoremstyle{descriptive}
\newtheorem{deff}[thm]{Definition}
\newtheorem{example}[thm]{Example}
\newtheorem{remark}[thm]{Remark}
\makeatletter
\renewenvironment{proof}[1][\proofname]{\par
  \pushQED{\qed}%
  \normalfont
  \trivlist
  \item[\hskip\labelsep
        \itshape
    #1\@addpunct{.}]\ignorespaces
}{%
  \popQED\endtrivlist\@endpefalse
}
\makeatother
\newcommand{\setm}{\setminus}
\renewcommand{\emptyset}{\varnothing}
\def\vint{\mathop{\mathchoice%
          {\setbox0\hbox{$\displaystyle\intop$}\kern 0.22\wd0%
           \vcenter{\hrule width 0.6\wd0}\kern -0.82\wd0}%
          {\setbox0\hbox{$\textstyle\intop$}\kern 0.2\wd0%
           \vcenter{\hrule width 0.6\wd0}\kern -0.8\wd0}%
          {\setbox0\hbox{$\scriptstyle\intop$}\kern 0.2\wd0%
           \vcenter{\hrule width 0.6\wd0}\kern -0.8\wd0}%
          {\setbox0\hbox{$\scriptscriptstyle\intop$}\kern 0.2\wd0%
           \vcenter{\hrule width 0.6\wd0}\kern -0.8\wd0}}%
          \mathopen{}\int}
\newcommand{\Cp}{{C_p}}
\DeclareMathOperator{\Div}{div}
\DeclareMathOperator{\capp}{cap}
\newcommand{\cp}{\capp_p}
\newcommand{\grad}{\nabla}
\DeclareMathOperator{\dist}{dist}
\DeclareMathOperator{\Lip}{Lip}
\newcommand{\Lipc}{{\Lip_c}}
\DeclareMathOperator{\interior}{int}
\DeclareMathOperator*{\essliminf}{ess\,lim\,inf}
\DeclareMathOperator*{\essinf}{ess\,inf}
\newcommand{\cpessinf}{\text{$\Cp$-}\essinf}
\newcommand{\bdry}{\partial}
\newcommand{\bdy}{\bdry}
\newcommand{\loc}{_{\rm loc}}
\newcommand{\simge}{\gtrsim}
\newcommand{\simle}{\lesssim}
\DeclareMathOperator{\Chcapp}{Ch-cap}
\newcommand{\cpp}{\Chcapp_p}
{\catcode`p =12 \catcode`t =12 \gdef\eeaa#1pt{#1}}      
\def\accentadjtext#1{\setbox0\hbox{$#1$}\kern   
                \expandafter\eeaa\the\fontdimen1\textfont1 \ht0 }
\def\accentadjscript#1{\setbox0\hbox{$#1$}\kern 
                \expandafter\eeaa\the\fontdimen1\scriptfont1 \ht0 }
\def\accentadjscriptscript#1{\setbox0\hbox{$#1$}\kern   
                \expandafter\eeaa\the\fontdimen1\scriptscriptfont1 \ht0 }
\def\accentadjtextback#1{\setbox0\hbox{$#1$}\kern       
                -\expandafter\eeaa\the\fontdimen1\textfont1 \ht0 }
\def\accentadjscriptback#1{\setbox0\hbox{$#1$}\kern     
                -\expandafter\eeaa\the\fontdimen1\scriptfont1 \ht0 }
\def\accentadjscriptscriptback#1{\setbox0\hbox{$#1$}\kern 
                -\expandafter\eeaa\the\fontdimen1\scriptscriptfont1 \ht0 }
\def\itoverline#1{{\mathsurround0pt\mathchoice
        {\rlap{$\accentadjtext{\displaystyle #1}
                \accentadjtext{\vrule height1.593pt}
                \overline{\phantom{\displaystyle #1}
                \accentadjtextback{\displaystyle #1}}$}{#1}}
        {\rlap{$\accentadjtext{\textstyle #1}
                \accentadjtext{\vrule height1.593pt}
                \overline{\phantom{\textstyle #1}
                \accentadjtextback{\textstyle #1}}$}{#1}}
        {\rlap{$\accentadjscript{\scriptstyle #1}
                \accentadjscript{\vrule height1.593pt}
                \overline{\phantom{\scriptstyle #1}
                \accentadjscriptback{\scriptstyle #1}}$}{#1}}
        {\rlap{$\accentadjscriptscript{\scriptscriptstyle #1}
                \accentadjscriptscript{\vrule height1.593pt}
                \overline{\phantom{\scriptscriptstyle #1}
                \accentadjscriptscriptback{\scriptscriptstyle #1}}$}{#1}}}}
\def\itunderline#1{{\mathsurround0pt\mathchoice
        {\rlap{$\underline{\phantom{\displaystyle #1}
                \accentadjtextback{\displaystyle #1}}$}{#1}}
        {\rlap{$\underline{\phantom{\textstyle #1}
                \accentadjtextback{\textstyle #1}}$}{#1}}
        {\rlap{$\underline{\phantom{\scriptstyle #1}
                \accentadjscriptback{\scriptstyle #1}}$}{#1}}
        {\rlap{$\underline{\phantom{\scriptscriptstyle #1}
                \accentadjscriptscriptback{\scriptscriptstyle #1}}$}{#1}}}}
\newcommand{\alp}{\alpha}
\newcommand{\al}{\alpha}
\newcommand{\ga}{\gamma}
\newcommand{\de}{\delta}
\newcommand{\eps}{\varepsilon}
\newcommand{\la}{\lambda}
\newcommand{\Om}{\Omega}
\renewcommand{\phi}{\varphi}
\newcommand{\p}{{$p\mspace{1mu}$}}
\newcommand{\R}{\mathbf{R}}
\newcommand{\limplus}{{\mathchoice{\vcenter{\hbox{$\scriptstyle +$}}}
  {\vcenter{\hbox{$\scriptstyle +$}}}
  {\vcenter{\hbox{$\scriptscriptstyle +$}}}
  {\vcenter{\hbox{$\scriptscriptstyle +$}}}
}}
\newcommand{\Np}{N^{1,p}}
\newcommand{\Nploc}{N^{1,p}\loc}
\newcommand{\uHp}{\itoverline{P}}   
\newcommand{\lHp}{\itunderline{P}}  
\newcommand{\Hp}{P}                 
\newcommand{\Hpind}[1]{P_{#1}}      
\newcommand{\oHp}{H}                
\newcommand{\oHpind}[1]{H_{#1}}     
\newcommand{\K}{\mathcal{K}}%
\newcommand{\ut}{\tilde{u}}
\newcommand{\ub}{\bar{u}}
\newcommand{\lQ}{\itunderline{Q}}
\newcommand{\lqq}{\underline{q}} 
\newcommand{\Ga}{\Gamma}
\newcommand{\Lploc}{L^p\loc}
\newcommand{\UU}{\mathcal{U}}%
\newcommand{\uP}{\itoverline{P}}     
\newcommand{\lP}{\itunderline{P}} 
\newcommand{\CPI}{C_{\rm PI}}

\newcounter{saveenumi}
\numberwithin{equation}{section}
\newcommand{\eqv}{\ensuremath{\mathchoice{\quad \Longleftrightarrow \quad}{\Leftrightarrow}}
                {\Leftrightarrow}{\Leftrightarrow}}
\newcommand{\imp}{\ensuremath{\mathchoice{\quad \Longrightarrow \quad}{\Rightarrow}
                {\Rightarrow}{\Rightarrow}} }
\newenvironment{ack}{\medskip{\it Acknowledgement.}}{}

\begin{document}

\authortitle{Anders Bj\"orn, Jana Bj\"orn
    and Juha Lehrb\"ack}
{Existence and almost uniqueness for \p-harmonic
Green functions  in metric spaces}
\title{Existence and almost uniqueness for\\ 
\p-harmonic Green functions \\
on bounded domains in metric spaces}
\author{
Anders Bj\"orn \\
\it\small Department of Mathematics, Link\"oping University, \\
\it\small SE-581 83 Link\"oping, Sweden\/{\rm ;}
\it \small anders.bjorn@liu.se
\\
\\
Jana Bj\"orn \\
\it\small Department of Mathematics, Link\"oping University, \\
\it\small SE-581 83 Link\"oping, Sweden\/{\rm ;}
\it \small jana.bjorn@liu.se
\\
\\
Juha Lehrb\"ack \\
\it\small Department of Mathematics and Statistics, University of Jyv\"askyl\"a,\\
\it\small P.O. Box 35 (MaD), FI-40014 University of Jyv\"askyl\"a, Finland\/{\rm ;}
\it \small juha.lehrback@jyu.fi
\\
}

\date{}

\maketitle

\noindent{\small
{\bf Abstract}.
We study (\p-harmonic) singular functions,
defined by means of upper gradients, 
in bounded domains in metric measure spaces.
It is shown that singular functions
exist if and only if the complement of the domain has positive
capacity, and that they
satisfy very precise capacitary identities for superlevel sets.
Suitably normalized singular functions are called Green functions.
Uniqueness of Green functions is largely an open problem beyond
unweighted $\R^n$, but we show that all Green functions (in a given domain
and with the same singularity) are comparable.
As a consequence, for \p-harmonic functions with a given pole
we obtain a similar comparison result near the pole.
Various characterizations of singular functions are also given.
Our results hold in complete metric spaces with a doubling 
measure supporting a \p-Poincar\'e inequality, or 
under similar local assumptions.
}

\bigskip

\noindent {\small \emph{Key words and phrases}:
capacitary potential,
doubling measure,
Green function,
metric space,
\p-harmonic function,
Poincar\'e inequality,
singular function.
}

\medskip

\noindent {\small Mathematics Subject Classification (2010):
Primary: 31C45; Secondary:  30L99,  31C15, 31E05, 35J92, 49Q20.
}

\section{Introduction}

Let $\Om \subset \R^n$ be a bounded domain, and let $x_0 \in \Om$.
Then $u$ is a \emph{\p-harmonic Green function} in $\Om$ 
with singularity at $x_0$ if
\begin{equation} \label{eq-Green-Rn}
\Delta_p u:= \Div(|\nabla u|^{p-2}\nabla u) = -\delta_{x_0}
\quad \text{in } \Om
\end{equation}
with zero boundary values on $\bdy\Om$ (in Sobolev sense),
where $\delta_{x_0}$ is the Dirac measure 
at $x_0$.
Such a Green function 
is in particular \p-harmonic in $\Omega\setminus\{x_0\}$
and \p-superharmonic in the whole domain $\Omega$.
If $1<p\le n$, it is also unbounded. 

In metric measure spaces,
Holopainen--Shan\-mu\-ga\-lin\-gam~\cite{HoSha}
gave a definition of \emph{singular functions}, 
which behave similarly 
to the Green functions in $\R^n$.
In this paper we introduce a simpler definition of 
singular functions, and then define Green functions as
suitably normalized singular functions.
See Section~\ref{sect-HS} for the definition 
from~\cite{HoSha} and for
a discussion on the relation between these different definitions.

In a metric measure space $X=(X,d,\mu)$
there is (a priori) no equation available for defining \p-harmonic
functions, and they are instead defined as local minimizers of the 
\p-energy integral
\[
     \int g_u^p \, d\mu,         
\]
where $g_u$ is the minimal \p-weak upper gradient of 
$u$, see Definition~\ref{deff-ug}.
This definition 
of \p-harmonic functions is
in, e.g., $\R^n$ equivalent to the 
definition using the \p-Laplace operator $\Delta_p u$.

\begin{deff} \label{deff-sing-intro}
Let $\Om \subset X$ be a bounded domain. 
A positive
function $u\colon\Om \to (0,\infty] $ is a \emph{singular function} in $\Om$
with singularity at $x_0 \in \Om$ if it satisfies the following
properties\/\textup{:} 

\addjustenumeratemargin{(S1)}
\begin{enumerate}
\renewcommand{\theenumi}{\textup{(S\arabic{enumi})}}%
\item \label{dd-s}
$u$ is \p-superharmonic in $\Om$\textup{;}
\item \label{dd-h}
$u$ is \p-harmonic in $\Om \setm \{x_0\}$\textup{;}
\setcounter{saveenumi}{\value{enumi}}
\item \label{dd-sup}
$u(x_0)=\sup_\Om u$\textup{;}
\item \label{dd-inf}
$\inf_\Om u = 0$\textup{;}
\item \label{dd-Np0}
$\ut \in \Nploc(X \setm \{x_0\})$, where
\[
   \ut = \begin{cases}
     u & \text{in } \Om, \\
     0 & \text{on } X \setm \Om.
     \end{cases}
\]
\end{enumerate}
\end{deff}
\medskip

There is actually some redundancy in this definition
under very mild assumptions,
see Theorem~\ref{thm-char-alt-intro} and
Remark~\ref{rmk-singular-assumptions}.
Singular functions are sometimes called Green functions in
the literature, and vice versa.
Moreover they can be normalized, or pseudonormalized,
in different ways.
For Green functions, we require the following precise normalization
in terms of the variational capacity of 
superlevel sets.

\begin{deff} \label{deff-Green-intro}
Let $\Om \subset X$ be a bounded domain. 
A \emph{Green function} 
is a singular function 
which satisfies 
\begin{equation} \label{eq-normalized-Green-intro-deff}
\cp(\Om^b,\Om) = b^{1-p},
\quad \text{when }
   0  <b < u(x_0),
\end{equation}
where $\Om^b=\{x\in \Om:u(x)\ge b\}$.
\end{deff}

In fact, it follows 
that Green functions $u$ satisfy 
\begin{equation} \label{eq-normalized-Green-intro}
\cp(\Om^b,\Om_a) = (b-a)^{1-p},
\quad \text{when }
   0 \le a <b \le  u(x_0),
\end{equation}
where 
$\Om_a=\{x\in \Om:u(x)>a\}$
and we interpret $\infty^{1-p}$ as $0$,
see Theorem~\ref{thm-normalized-sing}.

In unweighted $\R^n$, the study of singular and 
(\p-harmonic) Green functions with $p\ne2$ goes
back to Serrin~\cite{Ser}, \cite{Ser1965}.
On domains in weighted $\R^n$
(with a \p-admissible weight) the existence
of singular functions follows from 
Heinonen--Kilpel\"ainen--Martio~\cite[Theorem~7.39]{HeKiMa}.
(Instead of \ref{dd-Np0} they showed that condition \ref{b.2} 
in Theorem~\ref{thm-sing-char} holds,
but in view of Theorem~\ref{thm-sing-char}
this establishes the existence of singular functions in our sense.)

The classical \p-harmonic
Green functions defined by \eqref{eq-Green-Rn}
in unweighted Euclidean domains  
(and similarly for domains in weighted $\R^n$
with a \p-admissible weight) coincide
with the Green functions given by Definition~\ref{deff-Green-intro},
see Remark~\ref{rmk-classical-Green}.
Uniqueness of Green functions in 
unweighted
Euclidean domains 
was for $p\ne2$ established by
Kichenassamy--Veron~\cite{KichVeron} (see Section~\ref{sect-Green}),
but is not really known beyond that.
In particular,
it remains open in weighted $\R^n$.
However, 
Holopainen~\cite[Theorem~3.22]{Ho} proved uniqueness
in regular relatively compact domains in 
$n$-di\-men\-sion\-al Riemannian manifolds (equipped with their natural measures)
when $p=n$.
Moreover, in Balogh--Holopainen--Tyson~\cite{BHT}, uniqueness
was shown
for global $Q$-harmonic Green functions 
in Carnot groups of homogeneous dimension~$Q$.

In this paper we show the existence of singular functions and
also of Green functions satisfying the precise
normalization \eqref{eq-normalized-Green-intro-deff},
or equivalently \eqref{eq-normalized-Green-intro},
under the following standard assumptions on the metric measure space $X$;
see Section~\ref{sect-preli} for the relevant definitions.

\medskip

\emph{We make the following general assumptions in the theorems 
in the introduction\/\textup{:}
Let $1<p<\infty$ and assume that $X$ is a
complete metric space equipped with a doubling
measure $\mu$ 
supporting a \p-Poincar\'e inequality.
Let $\Om \subset X$ be a bounded domain and let $x_0 \in \Om$.}

\medskip

These assumptions are fulfilled in 
weighted $\R^n$ equipped with a \p-admissible measure,
on Riemannian  manifolds and Carnot--Carath\'eodory spaces
equipped with their natural measures,
and in many other situations,
see Sections~\ref{sect-preli} and~\ref{sect-Cheeger} for
further details.
Actually, the above assumptions on the space $X$ can be relaxed
to similar local assumptions.
The same applies also to our other results, 
see Section~\ref{sect-local-assumptions} for details.

The following theorem summarizes some of 
our main results.

\begin{thm} \label{thm-main-intro}
\quad 
\begin{enumerate}
\item \label{thm-intro-a}
There exists a Green function\/ \textup{(}or
equivalently, 
in view of \ref{thm-intro-b},  a singular 
function\/\textup{)}
in $\Om$ with singularity at $x_0$
if and only if $\Cp(X\setm \Om)>0$
\textup{(}which is always true if $X$ is unbounded\/\textup{)}.
\item \label{thm-intro-b}
If $u$ is a singular function in $\Om$ with singularity at $x_0$,
 then there is
a unique $\alp>0$ such that $\alp u$ is a Green function.
\item \label{thm-intro-c}
If $u$ and $v$ are two Green functions in $\Om$ with singularity at $x_0$,
then 
\begin{equation}   \label{eq-u-comp-v}
u \simeq v,
\end{equation}
where 
the comparison constants  depend only on $p$,
the doubling constant and the constants in the Poincar\'e inequality.
If moreover 
$\Cp(\{x_0\})>0$, then $u=v$
and it is a multiple of the capacitary potential for $\{x_0\}$ in $\Om$.
\end{enumerate}
\end{thm}

When $\Cp(\{x_0\})=0$,
Theorem~\ref{thm-main-intro}\,\ref{thm-intro-c} gives almost 
uniqueness of Green functions,
and in particular shows that all Green functions have the same
growth behaviour near the singularity.
As mentioned above, uniqueness of Green functions is not known even 
in weighted $\R^n$ (when $\Cp(\{x_0\})=0$).

The next result shows that
\eqref{eq-u-comp-v} is strong enough to make 
\p-harmonic functions into singular ones, provided that $\Cp(\{x_0\})=0$.

\begin{thm} \label{thm-sing-iff-comp-intro}
Assume that $\Cp(\{x_0\})=0$.
Let $u$ be a singular function in $\Om$ with singularity at $x_0$,
and let  $v\colon\Om \to (0,\infty]
$ be a function which is \p-harmonic 
in $\Om \setm \{x_0\}$.

Then $v$ is a singular function in $\Om$ with singularity at $x_0$
if and only if $v \simeq u$.
\end{thm}

Holopainen--Shan\-mu\-ga\-lin\-gam~\cite{HoSha}
provided a construction of singular functions
(according to their definition); see however
Remark~\ref{rmk-HS-existence}.
We show in Proposition~\ref{prop-HS-vs-us} 
that, under the assumptions used in~\cite{HoSha}, 
the definition therein
is essentially equivalent to Definition~\ref{deff-sing-intro}, up to
a normalization. Hence we also recover the existence
of singular functions according to the definition in~\cite{HoSha}.
Nevertheless, Definition~\ref{deff-sing-intro} seems
to be both more general and 
more flexible, and hence better suited e.g.\ for studying
the existence and uniqueness of
singular and Green functions. 
In particular, the definition in~\cite{HoSha} contains explicit
superlevel set inequalities, 
whereas we show in Lemma~\ref{lem-level-est}
that a precise \emph{superlevel set identity} is a
consequence of the properties assumed in Definition~\ref{deff-sing-intro}.
The absence of any a priori superlevel set requirements makes
it easy to apply our results to 
general \p-harmonic functions with poles,
see Theorem~\ref{thm-pharm-pole}.

From the superlevel set property we in turn
obtain the following pointwise estimate for 
Green functions near their
singularities.

\begin{thm} \label{thm-compare-bdy-ball}
If $u$ is a Green function in $\Om$ with singularity at $x_0$,
then for all $r>0$ such that $B_{50\lambda r}\subset\Om$
and all $x\in\bdry B_r$,
\begin{equation}\label{eq-compare-bdy-ball}
 u(x) \simeq \cp(B_r,\Om)^{1/(1-p)},
\end{equation}
where the comparison constants depend only on $p$,
the doubling constant and the constants in the Poincar\'e inequality.
Here $\la$ is the dilation constant in the \p-Poincar\'e inequality.
\end{thm}

In weighted $\R^n$ (with a \p-admissible weight),
\eqref{eq-compare-bdy-ball} was obtained by 
Heinonen--Kilpel\"ainen--Martio~\cite[Theorem~7.41]{HeKiMa};
for  $p=2$ it goes back to 
Fabes--Jerison--Kenig~\cite[Lemma~3.1]{FaJeKe}.
For \p-Laplacian-type equations of the form
\begin{equation} \label{eq-dvgA=B}
\Div A(x,u,\nabla u)=B(x,u,\nabla u)
\end{equation}
in unweighted
$\R^n$, with $1<p<\infty$, it
is due to Serrin~\cite[Theorem~12]{Ser}, \cite[Theorem~1]{Ser1965}.
In Carnot--Carath\'eodory spaces, \eqref{eq-compare-bdy-ball}
was proved by 
Capogna--Danielli--Garofalo~\cite[Theorem~7.1]{CaDaGa2}.
It was also obtained in some specific cases on metric spaces
by Danielli--Garofalo--Marola~\cite{DaGaMa},
see Remark~\ref{rmk-DGM}.
In~\cite[Section~6]{DaGaMa} they obtained some further 
results for Cheeger singular and Cheeger--Green functions,
cf.\ Section~\ref{sect-Cheeger}.
See also Holopainen~\cite[Section~3]{Ho}
for results on Green functions in 
regular relatively compact domains in $n$-dimensional Riemannian 
manifolds (equipped with their natural measures) when $1<p \le n$.

We also establish various useful characterizations for singular functions.
Theorems~\ref{thm-sing-iff-comp-intro}  
and~\ref{thm-char-alt-intro}
contain some of these, but 
in Sections~\ref{case-Cp(x0)=0}--\ref{sect-Green}
we obtain several additional characterizations, which are either 
more technical to state
or which only hold in one of the cases 
$\Cp(\{x_0\})=0$ or $\Cp(\{x_0\})>0$.

\begin{thm} \label{thm-char-alt-intro}
Assume that
$\Cp(X \setm \Om)>0$ and
let $u\colon\Om \to (0,\infty]$.
Then the following are equivalent\/\textup{:}
\begin{enumerate}
\item $u$ is a singular function in $\Om$ with singularity at $x_0$\textup{;}
\item
$u$ satisfies \ref{dd-s}, \ref{dd-h} and \ref{dd-Np0}\textup{;}
\item
$u(x_0)=\lim_{x \to x_0} u(x)$ and $u$ satisfies \ref{dd-h} and \ref{dd-Np0}.
\end{enumerate}
\end{thm}

The outline of the paper is as follows. 
We begin in Section~\ref{sect-preli} by recalling the basic definitions
related to the analysis on metric spaces. 
In Section~\ref{sect-cap-pot} we establish sharp
superlevel set formulas for capacitary potentials.
Such a formula 
was obtained in weighted $\R^n$ (with a \p-admissible weight) in 
Heinonen--Kilpel\"ainen--Martio~\cite[p.~118]{HeKiMa}.
Their argument depends on the Euler--Lagrange equation,
which is not available in the metric space setting considered here.
Nevertheless,  we are able to obtain this formula
with virtually no assumptions on the metric space nor on the sets
involved,
and at the same time the proof is
considerably shorter than the one in \cite[pp.\ 116--118]{HeKiMa}.
See Section~\ref{sect-cap-pot} for more details.

Section~\ref{sect-harm} contains a discussion about
(super)harmonic functions in the metric setting,
while in Section~\ref{sect-perron} we obtain,
with the help of harmonic extensions and Perron solutions, some finer
properties for these functions and, in particular,
for capacitary potentials. 

The actual study of singular and Green functions begins in Section~\ref{sect-singular},
where we record some easy observations concerning singular functions. 
Sections~\ref{case-Cp(x0)=0} and~\ref{case-Cp(x0)>0}
contain proofs for the existence and further properties
of singular functions under the respective assumptions
that $\Cp(\{x_0\})=0$ or $\Cp(\{x_0\})>0$.
Then, in Section~\ref{sect-Green}, we establish a sharp
superlevel set property for superharmonic functions and show how
this property yields the existence of Green functions.
In Section~\ref{sect-pole} we study
the growth behaviour of \p-harmonic functions with poles.
Local assumptions 
are
discussed in Section~\ref{sect-local-assumptions},
and in Section~\ref{sect-HS} we compare our definitions
and results with those in 
Holopainen--Shan\-mu\-ga\-lin\-gam~\cite{HoSha}.

By the theory of Cheeger~\cite{Cheeg}, it is possible to
use also a PDE approach to the study of singular and Green functions 
in metric spaces satisfying the
standard assumptions. In Section~\ref{sect-Cheeger} we show that   
in this setting the Cheeger--Green functions,
based on 
Definition~\ref{deff-Green-intro}, actually satisfy
an equation corresponding to~\eqref{eq-Green-Rn} and hence
the situation is analogous to that in (weighted) $\R^n$.
Note, however, that 
Cheeger \p-(super)harmonic functions,
and thus also the corresponding singular and Green functions,
differ in general 
from those defined by means of upper gradients. 

\begin{ack}
A.B. and J.B.  were supported by the Swedish Research Council,
grants 2016-03424 and 621-2014-3974, respectively.
J.L. was supported by 
the Academy of Finland, grant 252108.
\end{ack}

\section{Preliminaries}\label{sect-preli}

We assume throughout the paper that $1 < p<\infty$ 
and that $X=(X,d,\mu)$ is a metric space equipped
with a metric $d$ and a positive complete  Borel  measure $\mu$ 
such that $0<\mu(B)<\infty$ for all balls $B \subset X$. 
The $\sigma$-algebra on which $\mu$ is defined
is obtained by the completion of the Borel $\sigma$-algebra.
It follows that $X$ is separable.
To avoid pathological situations we assume that $X$ contains
at least two points.

Next we are going to introduce the
necessary background on Sobolev spaces and capacities in metric spaces.
Proofs of most of the results mentioned in this section
can be found in the monographs
Bj\"orn--Bj\"orn~\cite{BBbook} and
Heinonen--Koskela--Shanmugalingam--Tyson~\cite{HKSTbook}.

A \emph{curve} is a continuous mapping from an interval,
and a \emph{rectifiable} curve is a curve with finite length.
We will only consider curves which are nonconstant, compact
and 
rectifiable, and thus each curve can 
be parameterized by its arc length $ds$. 
A property is said to hold for \emph{\p-almost every curve}
if it fails only for a curve family $\Ga$ with zero \p-modulus, 
i.e.\ there exists $0\le\rho\in L^p(X)$ such that 
$\int_\ga \rho\,ds=\infty$ for every curve $\ga\in\Ga$.

We begin with the notion of \p-weak upper gradients 
as defined by 
Koskela--MacManus~\cite{KoMc}, 
see also Heinonen--Koskela~\cite{HeKo98}.

\begin{deff} \label{deff-ug}
A measurable function $g\colon X \to [0,\infty]$ is a \emph{\p-weak upper gradient}
of $f\colon X \to [-\infty,\infty]$ if for \p-almost every curve
$\gamma\colon  [0,l_{\gamma}] \to X$,
\[ 
         |f(\gamma(0)) - f(\gamma(l_{\gamma}))| \le \int_{\gamma} g\,ds,
\]
where we follow the convention that the left-hand side is $\infty$ 
whenever at least one of the 
terms therein is $\pm \infty$.
\end{deff}

If $f$ has an upper gradient in $\Lploc(X)$, then
it has an a.e.\ unique \emph{minimal \p-weak upper gradient} $g_f \in \Lploc(X)$
in the sense that for every \p-weak upper gradient $g \in \Lploc(X)$ of $f$ we have
$g_f \le g$ a.e.  
Following Shanmugalingam~\cite{Sh-rev}, 
we define a version of Sobolev spaces on the metric space $X$.

\begin{deff} \label{deff-Np}
For a measurable function $f\colon X\to [-\infty,\infty]$, let 
\[
        \|f\|_{\Np(X)} = \biggl( \int_X |f|^p \, d\mu 
                + \inf_g  \int_X g^p \, d\mu \biggr)^{1/p},
\]
where the infimum is taken over all \p-weak upper gradients of $f$.
The \emph{Newtonian space} on $X$ is 
\[
        \Np (X) = \{f: \|f\|_{\Np(X)} <\infty \}.
\]
\end{deff}
\medskip

The space $\Np(X)/{\sim}$, where  $f \sim h$ if and only if $\|f-h\|_{\Np(X)}=0$,
is a Banach space and a lattice.
In this paper we assume that functions in $\Np(X)$ are defined everywhere,
not just up to an equivalence class in the corresponding function space.
This is needed for the definition of \p-weak upper gradients to make sense.
For a measurable set $A\subset X$, the Newtonian space $\Np(A)$ is defined by
considering $(A,d|_A,\mu|_A)$ as a metric space in its own right. 
If $f,h \in \Nploc(X)$, then $g_f=g_h$ a.e.\ in $\{x \in X : f(x)=h(x)\}$.
In particular, $g_{\min\{f,c\}}=g_f \chi_{\{f < c\}}$ for any $c \in \R$.

\begin{deff}
The \emph{Sobolev capacity} of an arbitrary set $E\subset X$ is
\[
\Cp(E) =\inf_u\|u\|_{\Np(X)}^p,
\]
where the infimum is taken over all $u \in \Np(X)$ such that
$u\geq 1$ on $E$.
We say that a property holds \emph{quasieverywhere} (q.e.)\ 
if the set of points  for which it fails has Sobolev capacity zero.
\end{deff}

The capacity is the correct gauge 
for distinguishing between two Newtonian functions. 
If $u \in \Np(X)$, then $u \sim v$ if and only if $u=v$ q.e.
Moreover, 
if $u,v \in \Nploc(X)$ and $u= v$ a.e., then $u=v$ q.e.
Both the Sobolev and the variational capacity 
(defined below in Definition~\ref{deff-cp})
are countably subadditive.

\begin{deff}
For measurable sets $E \subset A \subset X$, let
\[
\Np_0(E;A)=\{f|_{E} : f \in \Np(A) \text{ and }
        f=0 \text{ on } A \setm E\}.
\]
If $A=X$, we omit $X$ in the notation and write $\Np_0(E)$.
Whenever convenient, we regard functions in $\Np_0(E;A)$ as extended
by zero to $A\setm E$.
\end{deff}

The measure
$\mu$ is \emph{doubling} if  there is a constant $C>0$ such that 
for all balls $B(x,r)=\{y \in X : d(x,y)<r\}$,
we have
\begin{equation}\label{eq:doubling-x0}
  \mu(B(x,2r))\le C \mu(B(x,r)).
\end{equation}

The space $X$ (or the measure $\mu$) supports a \emph{\p-Poincar\'e inequality} if
there exist constants $C>0$ and $\lambda \ge 1$
such that for all balls $B=B(x,r)\subset X$, 
all integrable functions $u$ on $X$, and all \p-weak upper gradients $g$ of 
$u$, 
\begin{equation}  \label{eq-deff-PI}
        \vint_{B} |u-u_B| \,d\mu
        \le C r \biggl( \vint_{\lambda B} g^{p} \,d\mu \biggr)^{1/p},
\end{equation}
where $u_B :=\vint_B u \,d\mu 
:= \int_B u\, d\mu/\mu(B)$ is the integral average
and $\la B$ stands for the dilated ball $B(x,\la r)$.

If $X$ is complete and $\mu$ is a doubling measure supporting
a \p-Poincar\'e inequality, then 
functions in $\Np(X)$ and those in $\Np(\Om)$, for open 
$\Om \subset X$, are \emph{quasicontinuous}. 
This will be important in Theorem~\ref{thm-Perron-fund},
but 
affects also how we formulate various statements,
such as the definition of the Sobolev capacity above.

If $X=\R^n$ is equipped with $d\mu=w\,dx$, then 
$w\ge0$ is a \p-admissible weight
in the sense of 
Heinonen--Kilpel\"ainen--Martio~\cite{HeKiMa}
if and only if $\mu$ is a doubling measure which
supports a \p-Poincar\'e inequality,
see
Corollary~20.9 in~\cite{HeKiMa} (which is only in the second edition)
and Proposition~A.17 in~\cite{BBbook}.
In this case,
$\Np(\R^n)$ and $\Np(\Om)$ are the 
refined Sobolev spaces 
defined in~\cite[p.~96]{HeKiMa},
and moreover our Sobolev and variational capacities
coincide with those in \cite{HeKiMa};
see Bj\"orn--Bj\"orn~\cite[Theorem~6.7\,(ix) and Appendix~A.2]{BBbook} and
Bj\"orn--Bj\"orn~\cite[Theorem~5.1]{BBvarcap}.
The situation is similar
on Riemannian  manifolds and Carnot--Carath\'eodory spaces
equipped with their natural measures; 
see Haj\l asz--Koskela~\cite[Sections~10 and~11]{HaKo2000}
and Section~\ref{sect-Cheeger} below for further details.

Throughout the paper, we write $Y \simle Z$ if there is an implicit
constant $C>0$ such that $Y \le CZ$.
We also write $Y \simge Z$ if $Z \simle Y$,
and $Y \simeq Z$ if $Y \simle Z \simle Y$.
Unless otherwise stated, we always allow the implicit comparison constants
to depend on the standard parameters, such as $p$, the doubling constant 
and the constants in the Poincar\'e inequality.

\section{Superlevel identities for capacitary potentials}
\label{sect-cap-pot}

\begin{deff} \label{deff-cp}
If 
$E \subset A$ are bounded subsets of $X$, then 
the \emph{variational capacity} of $E$ with respect to $A$ is
\begin{equation} \label{eq-cp-def}
\cp(E,A) = \inf_u\int_{X} g_u^p\, d\mu,
\end{equation}
where the infimum is taken over all $u \in \Np(X)$
such that $u\geq 1$ on $E$ and $u=0$ 
on  $X  \setm A$.
If no such function $u$ exists then $\cp(E,A)=\infty$.
\end{deff}

One can equivalently take the above infimum  over all 
$u \in \Np(X)$ such that
$u\geq 1$ q.e.\ on $E$ and $u=0$ q.e.\ on $X \setm A$;
we call such $u$ \emph{admissible} for the capacity
$\cp(E,A)$.

Since $A$ is not required to be measurable we cannot take
the integral in \eqref{eq-cp-def} over $A$, and
it is also important that the minimal \p-weak upper gradient
of $u$ is taken with respect to $X$. 
However, if $A$ is open then the integral 
and the minimal \p-weak upper gradient can equivalently
be taken over $A$.

\begin{deff} \label{deff-cap-pot}
Let $E\subset A$ be bounded subsets of $X$.
A \emph{capacitary potential} for the condenser $(E,A)$ is a minimizer 
for \eqref{eq-cp-def}, i.e.\
an admissible function 
realizing this infimum.
\end{deff}

Provided that $\cp(E,A)<\infty$, there is always
a minimizer $u$, i.e.\ a capacitary potential, by
Theorem~5.13 in Bj\"orn--Bj\"orn~\cite{BBnonopen};
this fact holds with no assumptions on the space.
If $\cp(E,A)= \infty$, there is no admissible function and hence
there cannot be any capacitary potential. 
Note that if $\dist(E,X \setm A)>0$, then $\cp(E,A)<\infty$.
Since $u$ is a minimizer, we have
\begin{equation}\label{eq-cap-pot}
\int_{X} g_{u}^p\,d\mu = \cp(E,A).
\end{equation}
Under rather mild assumptions, capacitary potentials
are unique up to sets of Sobolev capacity zero, see 
\cite[Theorem~5.13]{BBnonopen}.
For more about capacitary potentials, see also 
Lemmas~\ref{lem-representation-cap-pot} and~\ref{lem-char-cap-pot}
below and the comment preceding them.

One of the crucial ingredients in our estimates for 
Green functions is the
following capacity formula for superlevel sets of capacitary potentials.

\begin{thm} \label{thm-cap-pot}
Assume that $E \subset A$ are bounded
sets 
such that $\cp(E,A)<\infty$ and let 
$u$ be a capacitary potential of $(E,A)$.
Let $A_a=\{x\in A:u(x)>a\}$ and $A^a=\{x\in A:u(x)\ge a\}$.
Then
\begin{alignat*}{2}
\cp(A^b,A_a) &=\cp(A^b,A^a) =(b-a)^{1-p} \cp(E,A),  
   & \quad & \text{if } 0 \le a < b  \le 1, \\
\cp(A_b,A_a) &=\cp(A_b,A^a) =(b-a)^{1-p} \cp(E,A),  
& \quad & \text{if } 0 \le a < b  < 1. 
\end{alignat*}
\end{thm}

We reduce the proof of Theorem~\ref{thm-cap-pot} to the following
special cases.

\begin{lem}  \label{lem-pot-est}
Assume that $E \subset A$ are bounded
sets 
such that $\cp(E,A)<\infty$ and let 
$u$ be a capacitary potential of $(E,A)$.
Let $A_a=\{x\in A:u(x)>a\}$ and $A^a=\{x\in A:u(x)\ge a\}$.
Then 
\begin{alignat}{2}
\cp(A^a,A) & = a^{1-p} \cp(E,A),  & \quad & \text{if } 0 <a \le 1, 
\label{eq-A^a,A} \\
\cp(A_a,A) & = a^{1-p} \cp(E,A),  & \quad & \text{if } 0 <a < 1, 
\label{eq-A_a,A} \\
\cp(E \cap A^a,A^a) &= (1-a)^{1-p} \cp(E,A), & \quad & \text{if } 0 \le a < 1, 
\label{eq-E,A^a} \\
\cp(E \cap A_a,A_a) &= (1-a)^{1-p} \cp(E,A), & \quad & \text{if } 0 \le  a < 1.
\label{eq-E,A_a} 
\end{alignat}
Moreover, $u_1=\min\{u/a,1\}$ 
is a capacitary potential of both 
$(A^a,A)$ and $(A_a,A)$, while
$u_2=(u-au_1)/(1-a)$
is a  capacitary potential of 
$(E \cap A^a,A^a)$ and $(E \cap A_a,A_a)$, under the same conditions on $a$
as in \eqref{eq-A^a,A}--\eqref{eq-E,A_a}.
\end{lem}

The first identity \eqref{eq-A^a,A}
was obtained for open $A$ in weighted $\R^n$ (with a \p-admissible weight) in 
Heinonen--Kilpel\"ainen--Martio~\cite[p.~118]{HeKiMa}.
Their argument depends on the Euler--Lagrange equation,
which is not available in the metric space setting considered here.
Nevertheless, the weaker estimate 
\[
\cp(A^a,A) \simeq a^{1-p} \cp(E,A)
\]
was obtained for open $A$ in  metric spaces 
in Bj\"orn--MacManus--Shanmugalingam~\cite[Lemma~5.4]{BMS}
using  a variational approach.
Our proof is also based on the variational method, 
and still yields the exact identity in the metric space 
setting, with virtually no assumptions whatsoever on the metric space,
but is at the same time 
shorter than the proofs in \cite[pp.\ 116--118]{HeKiMa} and \cite{BMS}.

For open $A$ in
complete metric spaces equipped with a doubling measure supporting
a \p-Poincar\'e inequality,
the 
identities \eqref{eq-A^a,A} and \eqref{eq-A_a,A}
were recently obtained
in Aikawa--Bj\"orn--Bj\"orn--Shan\-mu\-ga\-lin\-gam~\cite{ABBSdich}
using similar ideas as here.

\begin{proof}[Proof of Lemma~\ref{lem-pot-est}]
The identities for $a=0$ and $a=1$ are rather immediate,
so assume that $0<a<1$.

Note that both $u_1=1$ and $u_2=1$ q.e.\ on $E$.
It follows that for each $t\in[0,1]$, 
the function $tu_1+(1-t)u_2$
is admissible in the definition of $\cp(E,A)$.
Since for a.e.\ $x\in X$, either $g_{u_1}=0$ or $g_{u_2}=0$,
we obtain (using also \eqref{eq-cap-pot}) that
\begin{equation} \label{eq-cp-fund-ineq}
\cp(E,A) =
\int_{X} g_u^p \,d\mu
     \le t^p \int_{X} g_{u_1}^p\,d\mu
                 + (1-t)^p \int_{X} g_{u_2}^p\,d\mu,
\end{equation}
with equality for $t=a$.
Denote the above integrals by $I$, $I_1$ and $I_2$, respectively.

If $u_1$ were not
a capacitary potential of $(A^a,A)$,
then we could replace $u_1$ by a capacitary potential $v$ of $(A^a,A)$
on the right-hand side above.
This would yield a strictly
smaller right-hand side when $t=a$, contradicting
the fact that we have equality throughout with $u_1$ on
the right-hand side when $t=a$.
Hence $u_1$ is a capacitary potential of $(A^a,A)$
and $I_1 = \cp(A^a,A)$.
Similarly,
$u_2$ is a capacitary potential of $(E \cap A_a,A_a)$ and  
$I_2 = \cp(E \cap A_a,A_a)$.

Next, we rewrite 
\eqref{eq-cp-fund-ineq} and the equality in it as
\begin{equation}
I \le t^p I_1 + (1-t)^p I_2
 \quad \text{and} \quad
I = a^p I_1 + (1-a)^p I_2.
         \label{estIwithI1I2}
\end{equation}
In particular, $t\mapsto t^p I_1 + (1-t)^p I_2$
attains its minimum for $t=a$.
Differentiating with respect to $t$ and letting $t=a$ we thus obtain
that $a^{p-1} I_1 = (1-a)^{p-1} I_2$.
Inserting this and $t=a$ into~\eqref{estIwithI1I2} yields
\begin{align*}
I &= a^p I_1 + a^{p-1}(1-a) I_1 = a^{p-1}I_1, \\ 
I &= a(1-a)^{p-1} I_2 + (1-a)^p I_2 = (1-a)^{p-1} I_2,
\end{align*}
proving 
\eqref{eq-A^a,A} and \eqref{eq-E,A_a}.

As $u=1$ q.e.\ on $E$, we see that
\begin{align*}
  \cp(E \cap A_a,A_a) 
  & \ge \cp(E \cap A_a,A^a) 
  =   \cp(E \cap A^a,A^a) \\
  & \ge \lim_{\eps \to 0\limplus} \cp(E \cap A^a,A_{a-\eps})
    =  \lim_{\eps \to 0\limplus} \cp(E \cap A_{a-\eps},A_{a-\eps}), 
\end{align*}
which together with \eqref{eq-E,A_a} shows
that \eqref{eq-E,A^a} holds.
The proof of \eqref{eq-A_a,A} is similar to the proof
of \eqref{eq-E,A^a}.
It also follows that $u_1$ and $u_2$ are capacitary potentials
of $(A_a,A)$ and 
$(E \cap A^a,A^a)$, respectively.
\end{proof}

\begin{proof}[Proof of Theorem~\ref{thm-cap-pot}]
We prove 
the identity for $\cp(A^b,A_a)$;
the other identities are shown similarly.
By Lemma~\ref{lem-pot-est}, $u_1=\min\{u/b,1\}$ 
is a capacitary potential  of $(A^b,A)$.
Since $u > a$ if and only if $u_1 > a/b$,
we get using first \eqref{eq-E,A_a}, with $E$ 
replaced by $A^b$, 
and then \eqref{eq-A^a,A} that
\begin{align*}
  \cp(A^b,A_a) 
  & = \Bigl(1-\frac{a}{b}\Bigr)^{1-p} \cp(A^b,A)\\
  & = \Bigl(1-\frac{a}{b}\Bigr)^{1-p} b^{1-p} \cp(E,A)
  =(b-a)^{1-p} \cp(E,A).
\qedhere
\end{align*}
\end{proof}

\section{\texorpdfstring{\p}{p}-harmonic and superharmonic functions}
\label{sect-harm}

\emph{From now on, but for 
Sections~\ref{sect-pole}--\ref{sect-HS}, 
 we assume that $X$ is complete, 
$\mu$ is doubling and supports a \p-Poincar\'e inequality,
$\Om\subset X$ is a nonempty open set,
and 
$x_0 \in \Om$ is a fixed point.
We also write
$B_r=B(x_0,r)$ for $r>0$.
As always in this paper, $1<p<\infty$.}

\medskip

Since $X$ is complete and $\mu$ is doubling, $X$ is also
proper,
i.e.\ bounded closed sets are compact.
It moreover
follows from the assumptions that
$X$ is quasiconvex (see e.g.\ \cite[Theorem~4.32]{BBbook}),
and thus connected and locally connected.
These facts will be important to keep in mind.
By Keith--Zhong~\cite[Theorem~1.0.1]{KeZh}, 
$X$ supports a $q$-Poincar\'e inequality for some $q<p$. 
This is assumed explicitly in some of the papers we refer to below.

In this section we recall the definitions of
\p-harmonic and superharmonic functions
and present some of their important properties that
will be needed later.
For proofs of the facts not proven in this section,
we refer to the monograph Bj\"orn--Bj\"orn~\cite{BBbook}.
The following definition of (super)minimizers is one of several 
equivalent versions 
in the literature,
cf.\ Bj\"orn~\cite[Proposition~3.2 and Remark~3.3]{ABkellogg}.

\begin{deff} \label{def-quasimin}
A function $u \in \Nploc(\Om)$ is a
\emph{\textup{(}super\/\textup{)}minimizer} in $\Om$
if 
\[ 
      \int_{\phi \ne 0} g^p_u \, d\mu
           \le \int_{\phi \ne 0} g_{u+\phi}^p \, d\mu
           \quad \text{for all (nonnegative) } \phi \in \Np_0(\Om).
\] 
A \emph{\p-harmonic function} is a continuous minimizer
(by which we mean real-valued continuous in this paper).
\end{deff}

It was shown in Kinnunen--Shanmugalingam~\cite{KiSh01} that 
under our standing assumptions,
a minimizer can be modified on a set of zero (Sobolev) capacity to obtain
a \p-harmonic function. For a superminimizer $u$,  it was shown by
Kinnunen--Martio~\cite{KiMa02} that its \emph{lsc-regularization}
\[ 
 u^*(x):=\essliminf_{y\to x} u(y)= \lim_{r \to 0} \essinf_{B(x,r)} u
\] 
is also a superminimizer  and $u^*= u$ q.e.

If $G$ is a bounded open set with $\Cp(X \setm G)>0$
and $f \in \Np(G)$,
then there is a unique \p-harmonic function 
$\oHpind{G} f$ in $G$ such that $ \oHpind{G} f -f \in \Np_0(G)$.
Whenever convenient, we let $\oHpind{G} f=f$ on $\bdy G$
or on $X\setm G$, 
provided that $f$ is defined therein.
The function $\oHpind{G} f$ is called the \emph{\p-harmonic extension} of $f$.
It is also the solution of the Dirichlet problem with boundary values
$f$ in the Sobolev sense.
An important property, coming from the ellipticity of the theory, is
the following comparison principle for $f_1,f_2 \in \Np(\itoverline{G})$, 
\begin{equation} \label{eq-H-comparison-principle}
\oHpind{G} f_1 \le     \oHpind{G} f_2 
\quad \text{whenever } 
 f_1 \le f_2 \text{ q.e. on } \bdy G,
\end{equation}
see Lemma~8.32 in \cite{BBbook}.

\begin{deff} \label{deff-superharm}
A function $u \colon \Om \to (-\infty,\infty]$ is 
\emph{superharmonic} in $\Om$ if
\begin{enumerate}
\renewcommand{\theenumi}{\textup{(\roman{enumi})}}%
\item \label{cond-a} $u$ is lower semicontinuous;
\item \label{cond-b} 
 $u$ is not identically $\infty$ in any component of $\Om$;
\item \label{cond-c}
for every nonempty open set $G \Subset \Om$ with $\Cp(X \setm G)>0$,
and all Lipschitz functions $v$ on $\itoverline{G}$,
we have $\oHpind{G} v \le u$ in $G$
whenever $v \le u$ on $\bdy G$.
\end{enumerate}
\end{deff}

As usual, by $G \Subset \Om$ we mean that $\itoverline{G}$
is a compact subset of $\Om$.
By Theorem~6.1 in Bj\"orn~\cite{ABsuper}
(or \cite[Theorem~14.10]{BBbook}),
this definition of superharmonicity is equivalent to  the 
definition usually used
in the Euclidean literature, e.g.\
in
Hei\-no\-nen--Kil\-pe\-l\"ai\-nen--Martio~\cite{HeKiMa}.

Superharmonic functions 
are always lsc-regularized (i.e.\ $u^*=u$).
Any lsc-regularized superminimizer is superharmonic, and
conversely any bounded superharmonic function is an lsc-regularized superminimizer.

The strong  minimum principle  for 
superharmonic functions, 
which says that a superharmonic function which attains
its minimum in a domain is constant therein,
holds by Theorem~9.13 in \cite{BBbook}.
The weak minimum principle says that if $G$ is a nonempty bounded open set,
and $u \in C(\itoverline{G})$ is superharmonic in $G$, then
$\min_G u = \min_{\bdy G} u$.
As $X$ is connected and complete,
the weak minimum principle follows from the strong one.

We will use the following extension property several times.
It is a direct consequence of
Theorems~6.2 and~6.3 in Bj\"orn~\cite{ABremove} 
(or Theorems~12.2 and~12.3 in~\cite{BBbook}).

\begin{lem}\label{lem-extension}
Let $x_0 \in \Om$ be such that $\Cp(\{x_0\})=0$. 
If $u\ge0$ is \p-harmonic in $\Om\setm \{x_0\}$,
then $u$ has a unique superharmonic extension to $\Om$,
given by $u(x_0):=\liminf_{x \to x_0} u(x)$.

If $u$ is in addition bounded from above or if 
$u\in\Np(\Om\setm\{x_0\})$, then the extension
is \p-harmonic in $\Om$.
\end{lem}

Also the following observation, 
containing a version of the Harnack inequality,
will be useful for us. 
It shows in particular that the $\liminf$ 
in Lemma~\ref{lem-extension}
is actually a true
limit.
Note that $\Cp(\{x_0\})>0$ is allowed here.

\begin{prop} \label{prop-punctured-ball}
Let $u \ge 0$ be a function which is superharmonic in $\Om$  and
\p-harmonic in $\Om\setm \{x_0\}$.
Then the limit
$a:=\lim_{x \to x_0} u(x)$ exists\/ {\rm(}possibly infinite\/{\rm)}
and $u(x_0)=a$. 

Moreover, if $0< \tau \le 1$ then
there is a constant $A>0$ which only depends
on $p$, $\tau$, the doubling constant of $\mu$ and
the constants in the \p-Poincar\'e inequality,
such that if $B=B_\rho$, $50 \la B \subset \Om$
and $K=\itoverline{B} \setm \tau B$, then 
\begin{equation} \label{eq-Harnack-spheres} 
             \max_{K} u \le A \min_{K} u = A \min_{\bdy B} u.
\end{equation}
\end{prop}

If $\Cp(\{x_0\})=0$, then by Lemma~\ref{lem-extension} we actually
do not need to require $u$ to be superharmonic in $\Om$,
only that $u(x_0)=\liminf_{x \to x_0} u(x)$;
the same is true for Proposition~\ref{prop-punctured-ball-char}.
But if $\Cp(\{x_0\})>0$ then superharmonicity cannot be omitted 
in general,
as seen by e.g.\ letting  $\Om=(-1,1) \subset \R$, $x_0=0$ and
$u=\chi_{(0,1)}$.

\begin{proof}
Let $G$ be the component of $\Om$ containing $x_0$.
Since $50\la B \subset \Om$, it follows from the Poincar\'e inequality 
that $B \subset G$, see
e.g.\ Lemma~4.10 in Bj\"orn--Bj\"orn~\cite{BBsemilocal}.
We start with the second part.
Let 
\[
m = \min_{K} u \quad \text{and} \quad M = \max_{K} u,
\]
which both exist and are finite as $u$ is \p-harmonic
(and thus continuous) in $\Om \setm \{x_0\}$. 
Fix $k>M$.
Then $u_k:=\min\{u,k\}$ is an lsc-regularized superminimizer in $\Om$.
By the weak minimum principle for superharmonic functions 
and the continuity of $u$, we see that
$m=\min_{\bdy B} u = \inf_{B}u=\inf_{B}u_k$.

Let $B'=B\bigl(y,\tfrac{1}{4}\tau\rho\bigr)$ be a ball with centre  $y \in K$
such that $M\le\sup_{B'}u_k$.
We shall now use the
weak Harnack inequalities from
Theorems~8.4 and 8.10 in Bj\"orn--Bj\"orn~\cite{BBbook}
(or Kinnunen--Shanmugalingam~\cite{KiSh01} and 
Bj\"orn--Marola~\cite{BMarola}).
Together with the doubling property of the measure $\mu$, they imply that
\[
    M \le \sup_{B'}u_k 
   \le C \biggl( \vint_{2B'}u_k^q \,d\mu\biggr)^{1/q}
   \le C' \biggl( \vint_{2B}u_k^q \,d\mu\biggr)^{1/q}
   \le A \inf_{B}u_k
   = A m,
\]
where $q>0$ is as in Theorem~8.10 in~\cite{BBbook} and the constants
$A$, $C$ and $C'$ 
depend only on $p$, $\tau$, the doubling constant of $\mu$ and
the constants in the \p-Poincar\'e inequality.
This proves~\eqref{eq-Harnack-spheres}.

To prove the first part of the proposition, let 
\[
m(r) = \min_{\bdry B_r} u \quad \text{and} \quad M(r) = \max_{\bdry B_r} u
\]
for $r < \rho$. 
As above, we have $m(r)=\inf_{B_r} u$, and so
$m(\,\cdot\,)$ is a nonincreasing function. Thus
$m_0=\lim_{r \to 0\limplus} m(r)$ exists. 

If $m_0=\infty$, then $\lim_{x \to x_0} u(x)=\infty$ and we are done.
Assume therefore
that $m_0<\infty$ and let 
$\eps >0$.
Then there is $r_1>0$ such that $m_0 - m(r_1) < \eps$.
Thus $v:=u-m(r_1)$
satisfies the assumptions of the
proposition with $\Om$ replaced by $B_{r_1}$.
We can thus use~\eqref{eq-Harnack-spheres} 
to obtain that for $0 < r < r_1/50\la$,
\begin{align*}
               M(r)- m_0 &\le 
               M(r)-m(r_1) =  \max_{\bdy B_r} v 
                  \le A \min_{\bdy B_r} v 
                  \\ &
               = A(m(r) -m(r_1)) \le A(m_0 -m(r_1)) < A \eps.
\end{align*}
Letting $\eps \to 0 \limplus$ shows that $\limsup_{x \to x_0} u(x)=m_0$,
and so $\lim_{x\to x_0} u(x)$ exists and equals $u(x_0)$
by the lower semicontinuity of $u$.
\end{proof}

The following characterization may be of independent interest.

\begin{prop} \label{prop-punctured-ball-char}
Assume that $\Cp(\{x_0\})=0$.
Let $u \ge 0$ be a function which is superharmonic in 
$\Om$ 
 and
\p-harmonic in $\Om\setm \{x_0\}$.
Then the following are equivalent\/\textup{:}
\begin{enumerate}
\item \label{m-a}
$u$ is \p-harmonic in $\Om$\textup{;}
\item \label{m-b}
$u$ is bounded in $B_r$ for some $r>0$\textup{;}
\item \label{m-c}
$u(x_0)<\infty$\textup{;}
\item \label{m-d}
$u \in \Np(B_r)$ for some $r>0$\textup{;}
\item \label{m-e}
$g_u \in L^p(B_r)$ for some $r>0$.
\end{enumerate}
\end{prop}

\begin{remark} \label{rmk-gu}
As $u$ is \p-harmonic in $\Om \setm \{x_0\}$
it belongs to $\Nploc(\Om \setm \{x_0\})$
and thus has a minimal \p-weak upper gradient $g_u \in L^p\loc(\Om \setm \{x_0\})$
in $\Om \setm \{x_0\}$.
Since $\Cp(\{x_0\})=0$, $g_u$ is also a \p-weak upper gradient
of $u$ within $\Om$, by Proposition~1.48 in \cite{BBbook}.
Even though it may happen that $g_u$ does not belong to $L^p\loc(\Om)$
it is still minimal in an obvious sense.
Thus $g_u$ is not as defined in Section~2.6 in \cite{BBbook},
but instead coincides with the minimal \p-weak upper gradient $G_u$ of
Section~5 in Kinnunen--Martio~\cite{KiMa}
and Section~2.8 in \cite{BBbook}.
In this paper, we will denote it by $g_u$ even within $\Om$.
This will, in particular, apply to singular and Green functions $u$.

The argument above, using Proposition~1.48 in \cite{BBbook},
also shows that $\Np(B_r)=\Np(B_r\setm \{x_0\})$
and thus \ref{m-d} can equivalently be formulated 
using $\Np(B_r\setm \{x_0\})$.
\end{remark}

It is not known if being \p-harmonic in a metric space
(defined using upper gradients as here) is a sheaf
property, see \cite[Open problems~9.22 and~9.23]{BBbook}.
This requires some care
when proving \ref{m-b} $\imp$ \ref{m-a} and \ref{m-d} $\imp$ \ref{m-a}
below.

\begin{proof}[Proof of Proposition~\ref{prop-punctured-ball-char}]
\ref{m-a} $\imp$ \ref{m-c} and 
\ref{m-a} $\imp$ \ref{m-d} 
These implications follow directly from the \p-harmonicity.

\ref{m-b} $\eqv$ \ref{m-c}
By Proposition~\ref{prop-punctured-ball},
$u(x_0)=\lim_{x \to x_0} u(x)$,
from which the equivalence follows.

\ref{m-b} $\imp$ \ref{m-a} and \ref{m-d} $\imp$ \ref{m-a}
Let $\Om_k=\{x \in B_k : \dist(x, X\setm \Om)>1/k\}$
(with the convention that $\dist(x,\emptyset)=\infty$).
If \ref{m-b} holds then, together with the \p-harmonicity
of $u$ in $\Om \setm \{x_0\}$, it shows that $u$ is
bounded in $\Om_k$.
If \ref{m-d} holds, 
we instead get that $u \in \Np(\Om_k \setm \{x_0\})$.
In both cases, it follows from
Lemma~\ref{lem-extension} that $u$ is \p-harmonic in $\Om_k$.
Hence $u$ is \p-harmonic in $\Om$,
by Propositions~9.18 and~9.21 in \cite{BBbook}.

\ref{m-d} $\imp$ \ref{m-e}
This is trivial.

\ref{m-e} $\imp$ \ref{m-d}
This follows from the $(p,p)$-Poincar\'e inequality
(see e.g.\ \cite[Corollary~4.24]{BBbook})
together with Proposition~4.13 in \cite{BBbook}.
\end{proof}

\begin{remark}\label{rem-zero-cap}
The distinction between the cases
$\Cp(\{x_0\})=0$ and $\Cp(\{x_0\})>0$ will often be important
in this paper. 
Hence we recall that (under our standing assumptions) 
Proposition~1.3 in   Bj\"orn--Bj\"orn--Lehrb\"ack~\cite{BBLeh1}
shows that $\Cp(\{x_0\})=0$ if 
\[
\liminf_{r\to 0} \frac{\mu(B_r)}{r^p}=0
\quad \text{or} \quad
\limsup_{r\to 0} \frac{\mu(B_r)}{r^p}<\infty.
\]
Conversely, if 
\[
\liminf_{r\to 0} \frac{\mu(B_r)}{r^q}>0
\quad \text{for some } q <p,
\]
then $\Cp(\{x_0\})>0$.
It is also shown in~\cite{BBLeh1} that the power of decay of $\mu(B_r)$
alone cannot determine whether $\Cp(\{x_0\})=0$.
However,  Proposition~5.3 
in our forthcoming paper~\cite{BBLehIntgreen} shows that
$\Cp(\{x_0\})=0$
if and only if 
 \[ 
 \int_0^{\delta} \biggl( \frac{\rho}{\mu(B_\rho)} \biggr)^{1/(p-1)} \,d\rho = \infty
\quad \text{for some (or equivalently all) $\de>0$.}
 \] 
\end{remark}

\section{Perron solutions and boundary behaviour}
\label{sect-perron}

\emph{In addition to the general assumptions from
the beginning of Section~\ref{sect-harm},
we assume in this section that $\Om$ is bounded and that $\Cp(X \setm \Om)>0$.}

\medskip

Perron solutions will be an important tool for us.

\begin{deff}\label{def:Perron}
Given $f\colon\bdy\Omega\to[-\infty,\infty]$, 
let $\UU_f(\Omega)$ be the collection of all superharmonic functions 
$u$ in $\Omega$ that are bounded from below and satisfy
\[
	\liminf_{\Omega\ni x\to y}u(x)
	\geq f(y)
	\quad\textup{for all }y\in\bdy\Omega.
\]
The \emph{upper Perron solution} of $f$ is defined by 
\[
	\uP_\Omega f(x)
	= \inf_{u\in\UU_f(\Omega)}u(x),
	\quad x\in\Omega.
\]
The lower Perron solution is defined 
similarly using subharmonic functions
or by $\lP_\Omega f = - \uP_\Omega f$.
If $\uP_\Omega f=\lP_\Omega f$, then we 
denote the common value by $\Hp_\Omega f$.
Moreover, if $\Hp_\Omega f$ is real-valued, then $f$ is said to be 
\emph{resolutive} (with respect to $\Omega$). 
\end{deff}

We will often write $\Hp f$ instead of $\Hp_\Omega f$,
and similarly for $\uP f$, $\lP f$ as well as for $\oHp f$.
An immediate consequence of Definition~\ref{def:Perron} is that
\[ 
\uHp f_1\leq\uHp f_2
\quad 
\text{whenever } f_1\le f_2 \text{ on } \bdy\Omega.
\] 

It follows from Theorem~7.2 in Kinnunen--Martio~\cite{KiMa02}
(or Theorem~9.39 in \cite{BBbook})
that $\lHp f\leq\uHp f$.
In each component of $\Omega$, $\uHp f$ is either \p-harmonic or 
identically $\pm\infty$,
by Theorem~4.1 in Bj\"orn--Bj\"orn--Shanmugalingam~\cite{BBS2}.
(This and all the facts below can also be found in 
Chapter~10 in \cite{BBbook}.)
We will need several 
results from \cite[Sections~5 and~6]{BBS2},
which we summarize as follows.
(Part~\ref{P-H=P} follows from \cite[Theorem~5.1]{BBS2}
after multiplying $f$ by a suitable Lipschitz cutoff function.)

\begin{thm} \label{thm-Perron-fund}
\begin{enumerate}
\item \label{P-H=P}
If $f \in \Np(G)$ for some open set $G \supset \overline{\Om}$,
then $\oHp f = \Hp f$.
\item \label{P-C}
If $f \in C(\bdy \Om)$, then $f$ is resolutive.
\item \label{P-C-uniq}
If $f$ is bounded and as in \ref{P-H=P} or \ref{P-C},
and $u$ is a bounded \p-harmonic function in $\Om$
such that
\[ 
      \lim_{\Om \ni x \to y} u(x)=f(y) 
      \quad
      \text{for q.e. } y \in \bdy \Om,
\] 
then $u=\Hp f$.
\end{enumerate}
\end{thm}

\begin{remark}
In order for \ref{P-H=P} to be possible it is important that
the Newtonian function $f$ is 
quasicontinuous, which follows from 
Theorem~1.1 in Bj\"orn--Bj\"orn--Shan\-mu\-ga\-lin\-gam~\cite{BBS5}
(or Theorem~5.29 in \cite{BBbook}).
\end{remark}

A boundary point $x_0 \in \bdy \Om$ is \emph{regular}
if $\lim_{\Omega\ni x\to x_0} Pf(x)=f(x_0)$ for every $f\in C(\bdy\Omega)$. 
We will need the following so-called \emph{Kellogg property}, see
Theorem~3.9 in Bj\"orn--Bj\"orn--Shanmugalingam~\cite{BBS}.
(The definition of regular points is different in \cite{BBS},
but by~\cite[Theorem~6.1]{BBS2} 
it is equivalent to our definition.)

\begin{thm} \label{thm-Kellogg}
\textup{(The Kellogg property)}
The set of irregular boundary points has capacity zero.
\end{thm}

We will also use that regularity is a local property of the boundary,
i.e.\ that $x_0 \in \bdy \Om$ is regular with respect to $\Om$
if and only if it is regular with respect to $\Om \cap B$
for every (or some) ball $B \ni x_0$,
see Theorem~6.1 in Bj\"orn--Bj\"orn~\cite{BB} 
(or \cite[Theorem~11.1]{BBbook}).
Moreover, if $G \subset \Om$ and $x_0 \in \bdy\Om \cap \bdy G$ 
is regular with respect to $\Om$, then it
is also regular with respect to $G$, see
\cite[Corollary~4.4]{BB} (or 
\cite[Corollary~11.3]{BBbook}).

Another important tool in this paper is capacitary potentials,
which we studied in Section~\ref{sect-cap-pot} 
in very general situations.
Under our standing assumptions we can say considerably more.
In particular,
capacitary potentials are unique up to sets of capacity zero,
by Theorem~5.13 in Bj\"orn--Bj\"orn~\cite{BBnonopen}.
In fact, it is easy to see that any capacitary potential is a solution
to the $\K_{\chi_E,0}(\Om)$-obstacle problem, as defined
  in \cite[Section~7]{BBbook}, and vice versa.
  Thus, provided that there is a capacitary potential of $(E,\Om)$,
Theorem~8.27 in \cite{BBbook} shows that there is a 
unique \emph{lsc-regularized capacitary potential} $u$,
i.e.\ such that $u^*=u$ in $\Om$ and $u \equiv 0$ on $X \setm\Om$.
Then $u$ also coincides with the ``capacitary potential'' as defined
in \cite[Definition~11.15]{BBbook},
and is therefore superharmonic in $\Om$, by \cite[Proposition~9.4]{BBbook}.
We shall sometimes call $u|_\Om$ a capacitary potential as well.
Recall that a  capacitary potential of $(E,\Om)$
exists if and only if $\cp(E,\Om)<\infty$.

We shall need the following two characterizations of capacitary 
potentials.

\begin{lem} \label{lem-representation-cap-pot}
  Let $E \subset \Om$ be relatively closed and 
  let $u \colon \Om \to [0,\infty]$.
Then $u$ is the lsc-regularized capacitary potential of $(E,\Om)$ if
and only if all of the following conditions hold\/{\rm:}
\begin{enumerate}
\item \label{r-1} 
$u$ is superharmonic in $\Om$\textup{;} 
\item \label{r-2}
$u$ is \p-harmonic in $G:=\Om \setm E$\textup{;}
\item \label{r-3} 
$u=1$ q.e.\ on $E$\textup{;}
\item \label{r-4}
$u \in \Np_0(\Om)$.
\end{enumerate}

Moreover, $u=\oHpind{G} u$ in $G$ and 
\begin{equation} \label{eq-reg-cap-pot}
\lim_{\Om \ni x \to y} u(x)=0
\quad \text{at every regular boundary
point $y\in\bdry\Om \setm \itoverline{E}$}.
\end{equation}
In particular, \eqref{eq-reg-cap-pot} holds for q.e.\
$y\in\bdry\Om \setm \itoverline{E}$.
\end{lem}

\begin{proof}
 If $u$ is an lsc-regularized capacitary potential of $(E,\Om)$,
then it satisfies \ref{r-3} and~\ref{r-4} by assumption, 
\ref{r-1} by the above, and \ref{r-2} by Theorem~8.28 in \cite{BBbook}.
Moreover, it is straightforward to see that within $G$,
$u$ is the lsc-regularized solution of the $\K_{0,u}(G)$-obstacle
problem, i.e.\ $u=\oHpind{G} u$ in $G$.
Hence, \eqref{eq-reg-cap-pot} and the last statement follow from 
\cite[Theorem~11.11\,(j)]{BBbook} together with the
Kellogg property (Theorem~\ref{thm-Kellogg}).

Conversely, if 
$u\in\Np_0(\Om)$ is \p-harmonic in $G$
then, by definition,  $u=\oHpind{G} u$ in $G$.
If, in addition, $u=1$ q.e.\ on $E$ then $u\in \K_{\chi_E,0}(\Om)$
and must therefore be a capacitary potential of $(E,\Om)$.
If it is also superharmonic in $\Om$,
then it is lsc-regularized. 
\end{proof}

\begin{lem}   \label{lem-char-cap-pot}
Let $K \subset \Om$ 
be compact and let
$u \colon \Om \to [0,\infty]$.
Then $u$ is the lsc-regularized capacitary potential of $(K,\Om)$ 
if and only if 
all of the following conditions hold\/{\rm:}
\begin{enumerate}
\item \label{m-1} 
$u$ is bounded and \p-harmonic in $G:=\Om\setm K$\textup{;} 
\item \label{m-2}
  $u\equiv1$ in $\interior K$\textup{;}
\item \label{m-3} 
$\displaystyle\lim_{G \ni x \to y} u(x) = \chi_K(y) 
          \text{ for q.e.\ } y \in \bdy G$\textup{;}
\item \label{m-4} 
$u(y) = \displaystyle\liminf_{G \ni x \to y} u(x) 
          \text{ for all } y \in \Om \cap \bdy K$. 
\end{enumerate}
Moreover, $u=\Hp_{G} \chi_K$ in $G$.
\end{lem}

\begin{proof}
Let $u$ be the lsc-regularized capacitary potential of $(K,\Om)$
and set
\[
  \ut=\begin{cases}
  u & \text{in } \Om \setm K,\\
  1 & \text{in } K, \\
  0 & \text{in } X \setm \Om.
  \end{cases}
\]
Then $\ut=u$ q.e.\ in $\Om$ and $\ut \in \Np_0(\Om)$. Thus
\[
 u = \oHpind{G} u 
     = \oHpind{G} \ut 
     = \Hpind{G} \ut 
     = \Hpind{G} \chi_K
\quad \text{in } G,
\]
by \eqref{eq-H-comparison-principle} and 
Theorem~\ref{thm-Perron-fund}\,\ref{P-H=P}.
(In particular, $\chi_K\in C(\bdy G)$ is resolutive with respect to $G$.)
Hence \ref{m-1} holds, and so does \ref{m-3} by the Kellogg property 
(Theorem~\ref{thm-Kellogg}).  
Since $u$ is the lsc-regularization of $\ut$, 
it satisfies \ref{m-2} and \ref{m-4}.

Conversely, if $u$ is bounded and \p-harmonic in $G$ and
satisfies \ref{m-3} then $u=\Hpind{G} \chi_K$ in $G$
by Theorem~\ref{thm-Perron-fund}\,\ref{P-C-uniq}. 
Hence, if $u$ also satisfies \ref{m-2} and~\ref{m-4}, 
then 
it is the lsc-regularized
capacitary potential of $(K,\Om)$,
by the first part of the proof.
\end{proof}

\begin{lem}   \label{lem-bdd-u-imp-Np1}
Assume that $K \subset \Om$ is compact and 
that $u\colon\Om \to  (0,\infty]$ is \p-harmonic in $\Om\setm K$.
For an open set $V\Subset \Om$ such that $K \subset V$,
consider the following conditions\/{\rm:}

\addjustenumeratemargin{(b.1$'$)}
\begin{enumerate}
\renewcommand{\theenumi}{\textup{(b.\arabic{enumi}$'$)}}%
\item \label{b.11}
$u \in \Np_0(\Om \setm V;X \setm V)$\textup{;}
\item \label{b.12}
$u$ is bounded in $\Om \setm V$ and
\begin{equation} \label{eq-qe-cond-bdy-lem1}
\lim_{\Om \ni x \to y} 
u(x)=0 \quad \text{for q.e.\ $y \in \bdy \Om$}\textup{;} 
\end{equation}
\item \label{b.13}
$u$ is bounded in $\Om \setm V$ and
$\min\{u,k\} \in \Np_0(\Om)$ for every $k>0$.
\end{enumerate}
Then \ref{b.13} $\imp$ \ref{b.11} $\eqv$ \ref{b.12}.
Moreover, \eqref{eq-qe-cond-bdy-lem1} can be equivalently replaced by
\begin{equation} \label{eq-qe-cond-bdy-lem1-alt}
\lim_{\Om \ni x \to y} 
u(x)=0 \quad \text{for every regular $y \in \bdy \Om$}.
\end{equation}
\end{lem}

As $u$ can be defined arbitrarily in $K$ in \ref{b.11} and \ref{b.12},
but not in \ref{b.13}, the implication \ref{b.13} $\imp$ \ref{b.11}
is not an equivalence.

\begin{proof}
Extend $u$ by letting $u=0$ on $X \setm \Om$.
Let $G=\Om \setm \overline{V}$.

\ref{b.13} \imp \ref{b.11}
Since $u$ is bounded in $\Om \setm V$, we have
$u=u_k:=\min\{u,k\}$ therein for large $k$. 
As $u_k \in \Np_0(\Om) \subset \Np_0(\Om \setm V;X \setm V)$, \ref{b.11} 
follows.

\ref{b.11} \imp \ref{b.12}
As $u$ is \p-harmonic in $\Om \setm K$
and $u \in \Np_0(\Om \setm V;X \setm V)$, 
it follows from the definition that  $H_{G} u=u$  in $G$.
Since $u$ is bounded on $\bdy V$
and vanishes on $\bdy \Om$,
there is $\alp >0$ such that $u \le \alp v$ on $\bdy G$,
where $v$ is the lsc-regularized capacitary 
potential for $\overline{V}$ in $\Om$.
By the comparison principle \eqref{eq-H-comparison-principle},
$u\le \al v$ also in $G$ and, in particular, $u$ is bounded therein.
Now, \eqref{eq-qe-cond-bdy-lem1-alt} follows from \eqref{eq-reg-cap-pot},
applied to $v$,
while \eqref{eq-qe-cond-bdy-lem1} follows
from \eqref{eq-qe-cond-bdy-lem1-alt}
and the Kellogg property (Theorem~\ref{thm-Kellogg}).

\ref{b.12} \imp \ref{b.11}
Let $\eta \ge 0$ be a Lipschitz function on $X$
such that $\eta=1$ on $\bdy V$
and $\eta=0$ in a neighbourhood of $K \cup (X \setm \Om)$.
As $u$ is \p-harmonic in $\Om \setm K$
and $\bdy V \Subset \Om \setm K$,
the function $u|_{\bdy G}= \eta u|_{\bdy G}$ is continuous.
Since \eqref{eq-qe-cond-bdy-lem1} or \eqref{eq-qe-cond-bdy-lem1-alt} holds, 
Theorem~\ref{thm-Perron-fund}\,\ref{P-C-uniq} shows that $u=P_{G} (\eta u)$.
It follows from the Leibniz rule (see \cite[Theorem~2.15]{BBbook}) 
that $\eta u \in \Np(X)$.
Hence Theorem~\ref{thm-Perron-fund}\,\ref{P-H=P}
implies that $u=H_G (\eta u)$ in $G$, which
in particular implies that $u \in \Np_0(\Om\setm V;X \setm V)$.
\end{proof}

Note that in the generality of Section~\ref{sect-cap-pot},
capacitary potentials are unique up to sets of capacity zero under rather
mild conditions,
by Theorem~5.13 in~\cite{BBnonopen}.
Nevertheless, it is far from clear if we can then always pick a canonical
representative in a suitable way.
In particular, even if $A$ is  open  
it is not at all clear
if $u^* = u$ q.e., i.e.\ whether
there always exists an lsc-regularized capacitary potential.
Under our standing assumptions in this section it is true
that $u^* =u$ q.e., but this is a consequence of the rather deep
interior regularity theory for superminimizers.

\section{Singular functions}
\label{sect-singular} 

\emph{In addition to the general assumptions from
the beginning of Section~\ref{sect-harm},
we assume in this section that $\Om$ is a bounded domain.}

\medskip

Recall properties \ref{dd-s}--\ref{dd-Np0} in
Definition~\ref{deff-sing-intro} of singular functions,
and that a \emph{domain} is a nonempty open connected set.
In this paper we are  interested in singular functions 
on bounded domains only.
For simplicity, we will often
just say that $u$ is a singular function, when we implicitly
mean within $\Om$ and with singularity $x_0$.

Note that a singular function must be nonconstant in $\Om$, 
as it is positive and \ref{dd-inf}  
holds. 
Our first observation, 
Proposition~\ref{prop-bdd-cp=0}, shows that $\Cp(X \setm \Om)>0$ is
a necessary condition for the existence of singular functions. 
(We will later show that it is also sufficient.)
Under this condition, the
theory of singular functions on bounded domains splits naturally
into two cases depending on whether $\Cp(\{x_0\})=0$ or $\Cp(\{x_0\})>0$,
which we will consider in Sections~\ref{case-Cp(x0)=0}
and~\ref{case-Cp(x0)>0}, respectively.
But first we deduce some results covering both cases simultaneously.

\begin{prop} \label{prop-bdd-cp=0}
If $\Cp(X \setm \Om)=0$, then there is  no singular function  in $\Om$
\textup{(}or more generally no positive superharmonic 
function in $\Om$ satisfying \ref{dd-inf}\textup{)}.
\end{prop}

\begin{proof}
It follows directly that $X$ must bounded.
Let $u>0$ be a 
superharmonic function in $\Om$.
By Theorem~6.3 in Bj\"orn~\cite{ABremove} 
(or Theorem~12.3 in~\cite{BBbook}),
$u$ has a superharmonic extension to all of $X$,
and by Corollary~9.14 in \cite{BBbook} this extension must be constant. 
Hence $u$ does not satisfy \ref{dd-inf} and
is, in particular, not a singular function.
\end{proof}

\begin{prop} \label{prop-singular=>p-harm}
If $\Cp(X\setm\Om)>0$ then there is no positive \p-harmonic
function in $\Om$ which satisfies \ref{dd-Np0}.
In particular, a singular function in $\Om$ is never 
\p-harmonic in all of $\Om$.
\end{prop}

\begin{proof}
Assume that $u$ is a positive \p-harmonic
function in $\Om$ satisfying \ref{dd-Np0}.
In particular, $u \in \Nploc(\Om)$. 
Extend $u$ as $0$ on $X \setm\Om$. 
Since $u \in \Nploc(\Om)$ 
and \ref{dd-Np0} holds, we see that $u \in \Np(X)$ 
and hence $u\in\Np_0(\Om)$.
But then $u =  \oHp u = \oHp 0 \equiv 0$ in $\Om$, 
which is a contradiction as $u$ is positive, 
i.e.\ no such function exists.

Finally, if there exists a singular function in $\Om$, then
Proposition~\ref{prop-bdd-cp=0} implies that $\Cp(X \setm \Om)>0$, and thus
the singular function cannot be \p-harmonic in $\Om$ by the first part 
of the lemma.
\end{proof}

\begin{remark} \label{rmk-singular-assumptions}
There is actually some redundancy in the definition of singular 
functions. 
As we shall see, by Theorem~\ref{thm-char-alt} below, if
$\Cp(X \setm \Om)>0$ then it is enough to assume that
$u$ satisfies \ref{dd-s}, \ref{dd-h} and \ref{dd-Np0}.
However, in the somewhat pathological case $\Cp(X \setm \Om)=0$, this is 
not enough as it would not prevent a
constant function from being a singular function.
To cover also this case it is enough to additionally assume \ref{dd-inf} or
to assume that $u$ is nonconstant, or that $u$ is not \p-harmonic
in $\Om$.

Even though \ref{dd-sup} is thus redundant, we have included it
in the definition as it seems such a natural
requirement for $u$.
Also, for unbounded domains it seems that one may need to require at least
these five properties to obtain a coherent 
theory of singular functions, but we
postpone such a study to a future paper.

That \ref{dd-s} cannot be dropped even if \ref{dd-sup} is replaced by
the stronger requirement

\medskip
%
%

{\addjustenumeratemargin{(S1$'$)}
\begin{enumerate}
\renewcommand{\theenumi}{\textup{(S\arabic{enumi}$'$)}}%
\setcounter{enumi}{\value{saveenumi}}
\item \label{dd-sup-var}
$u(x_0)=\sup_{\Om\setm \{x_0\}} u$,
\end{enumerate}
}\medskip

\noindent
follows by considering the function
\[ 
    u(x)=\begin{cases}
      1+x, & -1 < x <  0, \\
      2-2x, & 0 \le x < 1,
      \end{cases}
\] 
which is \p-harmonic 
in $(-1,1)\setm\{0\}\subset X:=\R$.
(Note that if \ref{dd-s} holds, then \ref{dd-sup} $\eqv$ \ref{dd-sup-var},
but without assuming \ref{dd-s}, assuming \ref{dd-sup-var} might
be more natural.)

However, if $\Cp(\{x_0\})=0$ then it follows from 
Theorem~\ref{thm-sing-char} below that \ref{dd-s} can be replaced
by e.g.\ $u(x_0)=\infty$, and thus
Proposition~\ref{prop-punctured-ball} shows that, in this case, 
\ref{dd-s} can be dropped provided that \ref{dd-sup} is kept.

To see that \ref{dd-h} cannot be dropped we instead let
$u$ be the lsc-regularized capacitary potential of
$(\itoverline{B}_1,B_2)$ in $\R^n$.
That \ref{dd-Np0} cannot be dropped follows from 
Example~\ref{ex-double-pole} below.
\end{remark} 

We conclude this section by summarizing some useful properties of
singular functions.

\begin{prop} \label{prop-existence-conseq}
If $u$ is a singular function in $\Om$ with singularity
at $x_0\in \Om$,
then
\begin{enumerate}
\item \label{ee-limit}
$u(x_0)=\lim_{x \to x_0} u(x)$\textup{;}
\item \label{ee-Np1}
$u \in \Np_0(\Om \setm B_r;X \setm B_r)$ 
for every $r>0$\textup{;}
\item \label{ee-Np2}
$\min\{u,k\} \in \Np_0(\Om)$ for every $k>0$\textup{;}
\item \label{ee-bdd}
$u$ is bounded in $\Om \setm B_r$ for every $r>0$\textup{;}
\item \label{ee-reg-0} 
$\lim_{\Om \ni x \to y} u(x)=0$ for q.e.\ $y \in \bdy \Om$, namely
for all $y \in \bdy \Om$ which are regular with respect to $\Om$.
\end{enumerate}
\end{prop}

Note that \ref{ee-Np1} is just an equivalent
way of writing \ref{dd-Np0}, when $\Om$ is bounded,
but not when $\Om$ is unbounded.
We therefore prefer to have the formulation \ref{dd-Np0}
in the definition.

\begin{proof}
\ref{ee-limit} This follows from Proposition~\ref{prop-punctured-ball}.

\ref{ee-Np1} As $\Om$ is bounded, \ref{ee-Np1} is equivalent 
to \ref{dd-Np0}.

\ref{ee-Np2}
Let $u_k=\min\{u,k\}$ which is a bounded superharmonic function,
and thus a superminimizer, and in particular $u_k \in \Nploc(\Om)$.
From \ref{ee-Np1} it then follows that $u_k \in \Np_0(\Om)$.

\ref{ee-bdd}  and \ref{ee-reg-0}
These follow from the already proven \ref{ee-Np1}
and Lemma~\ref{lem-bdd-u-imp-Np1} applied to $K=\{x_0\}$ and $V=B_r$,
together with \eqref{eq-qe-cond-bdy-lem1-alt}.
\end{proof}

\section{Characterizations when 
\texorpdfstring{$\Cp(\{x_0\})=0$}{}}
\label{case-Cp(x0)=0}

\emph{In addition to the general assumptions from
the beginning of Section~\ref{sect-harm},
we assume 
in Sections~\ref{case-Cp(x0)=0}--\ref{sect-Green} that
$\Om$ is a bounded domain such that 
$\Cp(X\setm \Om)>0$.
In particular, $\Cp(\bdy\Om)>0$ by\/ \textup{\cite[Lemma~4.5]{BBbook}}.
}
\medskip

As already mentioned, the
theory of singular functions (on bounded domains)
splits naturally
into the two cases $\Cp(\{x_0\})=0$ and $\Cp(\{x_0\})>0$.
We postpone the study of the latter case to Section~\ref{case-Cp(x0)>0}
and concentrate on the case 
$\Cp(\{x_0\})=0$ in this section.

Note first that, when $\Cp(\{x_0\})=0$,
it follows from the extension Lemma~\ref{lem-extension} that 
the requirement of superharmonicity in 
the definition of singular functions
can be replaced by the condition that $u(x_0)=\liminf_{x \to x_0} u(x)$.
In fact, by the following result, this also forces $u(x_0)=\infty$.

\begin{lem} \label{lem-existence-conseq=0}
Assume that $\Cp(\{x_0\})=0$.
Also assume that $u$ is a singular function in $\Om$ with singularity
at $x_0$, or more generally that $u\colon\Om \to (0,\infty]$ satisfies
\ref{dd-s}, \ref{dd-h} and \ref{dd-Np0} in 
Definition~\ref{deff-sing-intro}. 
Then 
$u(x_0)=\lim_{x \to x_0} u(x)=\infty$.
\end{lem}

That we only assume 
\ref{dd-s}, \ref{dd-h} and \ref{dd-Np0} 
will play a role in the proof of Theorem~\ref{thm-sing-char}.

\begin{proof}
We already know from Proposition~\ref{prop-punctured-ball}
that $u(x_0)=\lim_{x \to x_0} u(x)$.
If $u(x_0)$ were finite, then $u$ would be bounded in $\Om$,
and thus 
$u|_{\Om \setm \{x_0\}}$ would have a \p-harmonic
extension to $\Om$ by Lemma~\ref{lem-extension}.
But this contradicts 
Proposition~\ref{prop-singular=>p-harm}.
\end{proof}

Singular functions can be characterized in many ways.
Our aim is to have as simple and flexible criteria as possible.
Note that  $u$ is assumed to be positive,
and that condition~\ref{a.3} can always be guaranteed by redefining $u$
at $x_0$.

\begin{thm} \label{thm-sing-char}
Assume that
$\Cp(\{x_0\})=0$.
Let $u\colon\Om \to  (0,\infty]$  and
consider the following properties\/\textup{:}
\addjustenumeratemargin{(a.1)}
\begin{enumerate}
\renewcommand{\theenumi}{\textup{(a.\arabic{enumi})}}%
\item \label{a.1}  
$u$ is  superharmonic 
in $\Om$\textup{;}
\item \label{a.2}  
$u(x_0)=\lim_{x \to x_0} u(x)$\textup{;}
\item \label{a.3}
$u(x_0)=\liminf_{x \to x_0} u(x)$\textup{;}
\item \label{a.4}
$u(x_0)=\infty$\textup{;}
\end{enumerate}
and
\begin{enumerate}
\renewcommand{\theenumi}{\textup{(b.\arabic{enumi})}}%
\item \label{b.1}  
$u \in \Np_0(\Om \setm B_r;X \setm B_r)$ for every $r>0$\textup{;}
\item \label{b.2}  
$u$ is bounded in $\Om \setm B_r$ for every $r>0$, and
\begin{equation} \label{eq-qe-cond-bdy} 
\lim_{\Om \ni x \to y} 
u(x)=0 \quad \text{for q.e.\ $y \in \bdy \Om$}\textup{;} 
\end{equation}
\item \label{b.3}  
$u$ is bounded in $\Om \setm B_r$ for every $r>0$, and
$\min\{u,k\} \in \Np_0(\Om)$ for every $k>0$.
\end{enumerate}

Let $j \in \{1,2,3,4\}$ and $k \in \{1,2,3\}$.
Then $u$ is a singular function in $\Om$ with singularity at $x_0$
if and only if $u$ 
is \p-harmonic in $\Om \setm \{x_0\}$ and $u$ satisfies\/ 
\textup{(a.$j$)} and\/ \textup{(b.$k$)}.
\end{thm}

\begin{example} \label{ex-double-pole}
Let $x_0=0$, $x_1=(1,0,\ldots,0)$  and $\Om=B(0,2) \setm \{x_1\}$ 
in (unweighted) $\R^n$, $n \ge 3$, with $p=2$. 
Also let $v(x)=|x|^{2-n}+|x-x_1|^{2-n}$ and $u=v-\Hp v$,
where $Pv$ is the Perron solution in $\Om$.
Then, by 
linearity, $u$ is $2$-harmonic in $\Om \setm \{x_0\}$
and superharmonic in $\Om$.
In fact, $u$ satisfies \ref{dd-s}--\ref{dd-inf} in 
Definition~\ref{deff-sing-intro}, 
but not \ref{dd-Np0}.
It also satisfies \ref{a.1}--\ref{a.4}, but not \ref{b.1}--\ref{b.3}.
This shows, in particular, that the boundedness assumptions
in \ref{b.2} and \ref{b.3} cannot be dropped.
\end{example}

As $\Cp(\{x_0\})=0$,
conditions~\ref{b.1}--\ref{b.3} allow $u(x_0)$ to be arbitrary,
which shows that conditions~\ref{a.1}--\ref{a.4} cannot
be omitted.

\begin{proof}[Proof of Theorem~\ref{thm-sing-char}]
If $u$ is a singular function, then $u$ is \p-harmonic in $\Om \setm \{x_0\}$
and satisfies \ref{a.1} by assumption.
It further satisfies \ref{a.2}, \ref{a.3} and  \ref{b.1}--\ref{b.3} by 
Proposition~\ref{prop-existence-conseq},
and \ref{a.4} by Lemma~\ref{lem-existence-conseq=0}.

Conversely, assume that $u$ 
is \p-harmonic in $\Om \setm \{x_0\}$ and 
satisfies\/ 
\textup{(a.$j$)} and\/ \textup{(b.$k$)}
for some $j$ and $k$. Lemma~\ref{lem-bdd-u-imp-Np1} shows that 
\ref{b.3} \imp \ref{b.1} $\eqv$ \ref{b.2}.
The implication
\ref{a.2} \imp \ref{a.3} 
is trivial, 
while
\ref{a.3} \imp \ref{a.1}
holds 
by Lemma~\ref{lem-extension} since $\Cp(\{x_0\})=0$.

We postpone the case $j=4$, but otherwise,
regardless of the values of $j,k\in \{1,2,3\}$, we have 
shown that \ref{a.1}, \ref{b.1} and \ref{b.2} are satisfied.
Thus \ref{dd-s} and \ref{dd-h} are satisfied.
As \eqref{eq-qe-cond-bdy} holds and
$\Cp(\bdy \Om)>0$, we obtain \ref{dd-inf}.
Extending $u$ by $0$ on $X \setm \Om$ and 
letting $r\to0$ in \ref{b.1} yields \ref{dd-Np0}.
By Lemma~\ref{lem-existence-conseq=0},
$u(x_0)=\infty = \sup_{\Om} u$ and \ref{dd-sup} holds, which concludes
the proof that $u$ is a singular function. 

Finally, consider the case when $j=4$ and $k\in \{1,2,3\}$.
We have already shown that \ref{b.2} is satisfied.
Let 
\[
    \ut(x)=\begin{cases}
      u(x), & x \ne x_0, \\
      \liminf_{y \to x_0} u(y), & x=x_0.
      \end{cases}
\]
Then $\ut$ is \p-harmonic in $\Om \setm \{x_0\}$ and satisfies
\ref{a.3} and \ref{b.2}.
So by the already established cases,
$\ut$ is a singular function.
Lemma~\ref{lem-existence-conseq=0} shows that $\ut(x_0)=\infty$, i.e.\
$u=\ut$ is a singular function.
\end{proof}

We are now prepared to prove the existence of singular functions
at points having zero capacity.

\begin{thm} \label{thm-existence}
If $\Cp(\{x_0\})=0$, then 
there is a singular function in $\Om$ with singularity at $x_0$.
\end{thm}

\begin{proof}
Let $r_0>0$ be so small that ${B}_{r_0}\Subset\Om$.
For $0<r\le r_0$, let $u_r$ be the 
lsc-regularized capacitary potential for $\itoverline{B}_r$ in $\Om$.
Then $u_r$ is superharmonic in $\Om$ 
and \p-harmonic in $\Om \setm \itoverline{B}_r$, 
by Lemma~\ref{lem-representation-cap-pot}.

Let $M_r=\max_{\bdy B_{r_0}} u_r >0$,
which exists by the continuity of $u_r$ in $\Om\setm\itoverline{B}_r$
(while $M_{r_0}=1$ as $\Cp(\bdy B_{r_0})>0$).
Also, $M_r>0$ by the strong minimum principle for 
superharmonic functions since $\Cp(B_r)>0$. 
Let $v_r = u_r/M_r$. Then $\max_{\bdy B_{r_0}} v_r=1$.
Thus we can use Harnack's convergence theorem 
(Proposition~5.1 in Shanmugalingam~\cite{Sh-conv} or
Theorem~9.37 in \cite{BBbook}) to find a subsequence 
$\{v_{r_j}\}_{j=1}^\infty$ converging locally uniformly in $\Om \setm \{x_0\}$
to a nonnegative \p-harmonic function $u$.
As $\Cp(\{x_0\})=0$, Lemma~\ref{lem-extension} implies that $u$ has
a superharmonic extension to $\Om$ given by
$u(x_0):=\liminf_{x \to x_0} u(x)$.
Clearly $u \le 1$ on ${\bdy B_{r_0}}$, and from the local uniform convergence and
the compactness of ${\bdy B_{r_0}}$ we conclude that $\max_{\bdy B_{r_0}} u=1$.
Thus $u$ is positive in $\Om$ by the strong minimum principle for 
superharmonic functions. 

By definition and the comparison principle~\eqref{eq-H-comparison-principle},
\[
v_r = \oHpind{G} v_r \le \oHpind{G} u_{r_0} = u_{r_0} 
\quad \text{in } G :=\Om\setm\itoverline{B}_{r_0}
\]
for all $0<r\le r_0$, and hence $0 \le u \le u_{r_0}$ in $G$.
Thus, by Lemma~\ref{lem-representation-cap-pot},
\[
   0 \le \liminf_{\Om \ni x \to y} u(x) 
   \le \limsup_{\Om \ni x \to y} u(x)
   \le \lim_{\Om \ni x \to y} u_{r_0}(x)=0 \quad 
\text{for q.e.\ } y \in \bdy \Om,
\]
i.e.\ \eqref{eq-qe-cond-bdy} holds.
Since $u$ is \p-harmonic, and thus continuous, in $\Om \setm \{x_0\}$,
it is bounded in the compact set $\itoverline{B}_{r_0}\setm B_r$
for every $r>0$.
As also $0 \le u \le 1$ in $G=\Om\setm\itoverline{B}_{r_0}$,
we see that $u$ is bounded in $\Om \setm B_r$ for every $r>0$.

We have thus shown that $u$ is a positive \p-harmonic 
function in $\Om \setm \{x_0\}$, 
which satisfies~\ref{a.1} and~\ref{b.2}, and hence 
$u$ is a singular function by Theorem~\ref{thm-sing-char}.
\end{proof}

\begin{remark}
In the above proof  we constructed a singular function using
capacitary potentials of balls. 
This is just for convenience, but there
is nothing special about balls in this case.
Indeed, if we let
$G_1 \supset G_2 \supset \ldots$ be open sets such that
$G_1 \Subset \Om$ and $\bigcap_{k=1}^\infty G_k = \{x_0\}$,
then we can instead use the capacitary potentials for~$\itoverline{G}_k$.
It is an open question, even in weighted $\R^n$
(with a \p-admissible weight),
 whether all such constructions lead to the
same singular function (upon proper normalization as in 
\eqref{eq-normalized-Green-intro-deff}).
\end{remark}

We conclude this section with a simple nonintegrability result for
singular functions. 
Part \ref{n-c}  is mainly interesting as contrasting
Proposition~\ref{prop-sing-char>0-extra} below, 
see also Theorem~\ref{thm-char-Cp=0}. 
In our forthcoming paper~\cite{BBLehIntgreen}, we will give much more
precise results on the $L^t$ integrability and nonintegrability 
of $u$ and $g_u$ for singular and Green functions $u$, where $t>0$.

\begin{prop} \label{prop-Np0-Cp=0}
Assume that $\Cp(\{x_0\})=0$
and that $u$ is a singular function in $\Om$ with singularity at $x_0$.
Extend $u$ by letting $u=0$ on $X \setm \Om$.
Then the following are true\/\textup{:}
\begin{enumerate}
\item \label{n-a}
$u \notin \Np(B_r)$ is true for every $r>0$\textup{;}
\item \label{n-b}
$\int_{B_r} g_u^p \, d\mu = \infty$ for every $r>0$\textup{;}
\item \label{n-c}
$u \notin \Np_0(\Om)$.
\end{enumerate}
\end{prop}

\begin{proof}
Parts~\ref{n-a} and~\ref{n-b} follow directly from
Proposition~\ref{prop-singular=>p-harm}
or
Lemma~\ref{lem-existence-conseq=0},
together with Proposition~\ref{prop-punctured-ball-char}.
Part \ref{n-c} then follows directly from \ref{n-a}.
\end{proof}

\section{Characterizations when 
\texorpdfstring{$\Cp(\{x_0\})>0$}{}}
\label{case-Cp(x0)>0}

\emph{Recall the standing assumptions from the beginning of 
Section~\ref{case-Cp(x0)=0}.}

\medskip

We now turn to the case when the singularity point $x_0$ 
has positive capacity.
As we shall see, singular functions are unique in this case, 
up to multiplication by positive constants.
By Theorem~\ref{thm-uniq>0} below,
there is also an explicit representative for singular functions, 
namely the capacitary potential for $\{x_0\}$ in $\Om$.

\begin{lem} \label{lem-lim-in-R}
Assume that $\Cp(\{x_0\})>0$, and let $u$ be a \p-harmonic function
in $\Om \setm \{x_0\}$.
Then $\liminf_{x \to x_0} u(x)<\infty$.

In particular, if $\lim_{x \to x_0} u(x)=:u(x_0)$ exists,
then $u(x_0) \in \R$.
\end{lem}

\begin{proof}
If $\liminf_{x \to x_0} u(x)=\infty$, then 
there is a connected open neighbourhood $G  \subset \Om$ of  $x_0$
such that $u >0$ in $G\setm \{x_0\}$.
The definition of Perron solutions implies that
$u/k\ge \Hpind{G \setm \{x_0\}} \chi_{\{x_0\}}$ for all $k>0$.
Letting $k \to \infty$ shows that 
$\Hpind{G \setm \{x_0\}} \chi_{\{x_0\}}\equiv0$, which contradicts 
$\Cp(\{x_0\})>0$ and the Kellogg property (Theorem~\ref{thm-Kellogg}).
Hence $\liminf_{x \to x_0} u(x)<\infty$.

Applying this also to $-u$ shows that when 
$u(x_0):=\lim_{x \to x_0} u(x)$ exists it must be real.
\end{proof}

The following is an existence and uniqueness result (up to normalization)
for singular functions when $\Cp(\{x_0\})>0$.

\begin{thm} \label{thm-uniq>0}
Assume that
$\Cp(\{x_0\})>0$, and let 
$v$ be the lsc-regularized capacitary potential for $\{x_0\}$ in $\Om$.
Then a function $u$ is a singular function in $\Om$ with singularity at $x_0$
if and only if there is a constant $0< b< \infty$ such that $u=bv$
in $\Om$.
Moreover, $b=u(x_0)=\lim_{x \to x_0} u(x)$ in that case.

In particular, $v$
is a singular function in $\Om$ with singularity at $x_0$.
\end{thm}

\begin{proof}
Let $u=bv$.
By definition, $u$ is nonnegative and bounded.
Lemma~\ref{lem-representation-cap-pot} shows that $u$ is \p-harmonic 
in $\Om \setm \{x_0\}$ and superharmonic in $\Om$.
As $\Cp(\bdy \Om)>0$ and $\Cp(\{x_0\})>0$, we conclude from 
Lemma~\ref{lem-char-cap-pot}\,\ref{m-3} that $\inf_\Om u =0$ and 
$u(x_0)= b= \sup_\Om u$.
In particular,
$u \not \equiv 0$, and thus $u >0$ in $\Om$ by the strong
minimum principle for superharmonic functions. 
Thus, $u$ is a singular function.

Conversely assume that $u$ is a singular function. 
By Proposition~\ref{prop-punctured-ball} and Lem\-ma~\ref{lem-lim-in-R}, 
$b:=u(x_0)=\lim_{x \to x_0} u(x)<\infty$.
Thus $u$ is a bounded superharmonic function in some
neighbourhood $B_r$ of $x_0$, and in particular $u \in \Np(B_{r/2})$.
Together with \ref{dd-Np0} this shows that $u \in \Np_0(\Om)$ and
Lemma~\ref{lem-representation-cap-pot} implies that $u=bv$ in $\Om$.
\end{proof}

Also 
when $\Cp(\{x_0\})>0$, singular functions can be characterized in many ways.

\begin{thm} \label{thm-sing-char>0}
Assume that
$\Cp(\{x_0\})>0$.
Let $u\colon \Om \to (0,\infty]$  and consider
the properties \textup{(a.$j$)} and \textup{(b.$k$)} from 
Theorem~\ref{thm-sing-char}.

Let $j \in \{1,2\}$ and $k \in \{1,2,3\}$.
Then $u$ is a singular function in $\Om$ with singularity at $x_0$
if and only if $u$ 
is \p-harmonic in $\Om \setm \{x_0\}$ and $u$ satisfies\/ 
\textup{(a.$j$)} and\/ \textup{(b.$k$)}.
\end{thm}

Note that compared with 
Theorem~\ref{thm-sing-char} (for the case when $\Cp(\{x_0\})=0$)
conditions \ref{a.3} and \ref{a.4} are omitted here.
By Theorem~\ref{thm-uniq>0}, condition \ref{a.4} is never satisfied
for singular functions when $\Cp(\{x_0\})>0$, so it cannot
be included here.
To see that \ref{a.3} cannot be included, 
consider the function
\[
    u(x)=\begin{cases}
      1+x, & -1 < x \le  0, \\
      2-2x, & 0 < x < 1,
      \end{cases}
\]
which is \p-harmonic 
in $(-1,1)\setm\{0\}\subset X:=\R$
and satisfies \ref{a.3}, \ref{b.2} and \ref{b.1},
but not \ref{a.2},
and hence not \ref{a.1} either, by Proposition~\ref{prop-punctured-ball}.
In particular, $u$ is not a singular function.
Also \ref{b.3} fails as functions in $\Np(\R)$ are continuous.

The above $u$
also shows that (a.$j$) cannot be dropped if $k\in \{1,2\}$.
We do not know if (a.$j$) is redundant when \ref{b.3} is assumed.
That \ref{b.1}--\ref{b.3} cannot be dropped follows by considering
the constant function $u \equiv 1$.

\begin{proof}[Proof of Theorem~\ref{thm-sing-char>0}]
If $u$ is a singular function, then $u$ is \p-harmonic in $\Om \setm \{x_0\}$
and satisfies \ref{a.1}  by assumption.
It further satisfies \ref{a.2} and \ref{b.1}--\ref{b.3} by 
Proposition~\ref{prop-existence-conseq}.

Conversely, assume that $u$ 
is \p-harmonic in $\Om \setm \{x_0\}$ and $u$ satisfies\/ 
\textup{(a.$j$)} and\/ \textup{(b.$k$)}
for some $j \in \{1,2\}$ and $k \in \{1,2,3\}$.
Lemma~\ref{lem-bdd-u-imp-Np1} shows that 
\ref{b.3} \imp \ref{b.1} $\eqv$ \ref{b.2}.

If \ref{a.2} holds, then $u(x_0) =\lim_{x\to x_0}u(x)<\infty$
by Lemma~\ref{lem-lim-in-R}.
Hence, in view of \ref{b.2},  $u$ is bounded in $\Om$.
Lemma~\ref{lem-char-cap-pot}, together with \ref{a.2} 
and \eqref{eq-qe-cond-bdy} from \ref{b.2},
implies that $u=u(x_0) v$, where $v$ is the lsc-regularized
capacitary potential for $\{x_0\}$ in $\Om$.
In particular, $u$ is superharmonic in $\Om$, and thus \ref{a.2} \imp \ref{a.1}.

Hence, regardless of the values of $j$ and $k$,
we have shown that \ref{a.1}, \ref{b.1} and \ref{b.2} hold,
and so \ref{dd-s}, \ref{dd-h} and \ref{dd-Np0} are satisfied.
As \eqref{eq-qe-cond-bdy} holds and $\Cp(\bdy \Om)>0$,
we obtain \ref{dd-inf}.

It remains to show that \ref{dd-sup} holds. 
If $u(x_0)$ were $\infty$ then this
would be immediate, so we may assume that $u(x_0) <\infty$. 
Proposition~\ref{prop-punctured-ball} implies that 
$\lim_{x\to x_0} u(x)=u(x_0)$ and hence 
$u=P_{\Om \setm \{x_0\}}  (u(x_0) \chi_{\{x_0\}})$, 
by \eqref{eq-qe-cond-bdy}  
and Theorem~\ref{thm-Perron-fund}\,\ref{P-C-uniq}.
Thus $u \le u(x_0)$ in $\Om$, and hence \ref{dd-sup} holds.
\end{proof}

The following result shows that 
if we strengthen \ref{b.1} in a suitable way, 
we do not even need to assume \ref{a.1} or \ref{a.2}.

\begin{prop} \label{prop-sing-char>0-extra}
Assume that
$\Cp(\{x_0\})>0$.
Then $u\colon \Om \to (0,\infty]$  
is a singular function in $\Om$ with singularity at $x_0$
if and only if 
$u$ 
is \p-harmonic in $\Om \setm \{x_0\}$ and $u \in \Np_0(\Om)$.
\end{prop}

Proposition~\ref{prop-Np0-Cp=0} shows that
the corresponding characterization is false when $\Cp(\{x_0\})=0$.
It also shows that condition \ref{b.1} cannot be 
replaced by assuming that $u \in \Np_0(\Om)$ in 
Theorem~\ref{thm-sing-char}, nor in
Theorem~\ref{thm-char-alt} below.

\begin{proof}
If $u$ is a singular function, then $u$ is \p-harmonic in $\Om \setm \{x_0\}$
and $u \in \Np_0(\Om)$, 
by
Theorem~\ref{thm-uniq>0}.

Conversely, assume that  $u \in \Np_0(\Om)$
is \p-harmonic in $\Om \setm \{x_0\}$.
As $\Cp(\{x_0\})>0$, we have $u(x_0) <\infty$, by 
\cite[Proposition~1.30]{BBbook}.
By definition, $u=\oHpind{\Om \setm \{x_0\}} u$, and
Lemma~\ref{lem-representation-cap-pot} implies that
\[
u=\oHpind{\Om \setm \{x_0\}} u = \oHpind{\Om \setm \{x_0\}} (u(x_0) v) 
= u(x_0) v,
\]
where $v$ is  the lsc-regularized capacitary potential for 
$\{x_0\}$ in $\Om$.
By Theorem~\ref{thm-uniq>0}, $u$ is a singular function in $\Om$.
\end{proof}

We conclude this section by 
summarizing which characterizations are true in both cases
$\Cp(\{x_0\})=0$ and $\Cp(\{x_0\})>0$.
Here \ref{h-d} is added compared with Theorem~\ref{thm-char-alt-intro}.
Recall that $\Om$ is bounded and $\Cp(X \setm \Om)>0$ in this section.

\begin{thm} \label{thm-char-alt}
Let $u\colon\Om \to (0,\infty]$, $j \in \{1,2\}$ and $k \in \{1,2,3\}$.
Assume that $u$ is \p-harmonic in $\Om \setm \{x_0\}$.
Then the following are equivalent\/\textup{:}
\begin{enumerate}
\item 
$u$ is a singular function in $\Om$ with singularity at $x_0$\textup{;}
\item
$u$ satisfies \ref{dd-s} and \ref{dd-Np0}\textup{;}
\item
$u(x_0)=\lim_{x \to x_0} u(x)$ and $u$ satisfies 
\ref{dd-Np0}\textup{;}
\item \label{h-d}
$u$ satisfies\/
\textup{(a.$j$)} and\/ \textup{(b.$k$)}
from Theorem~\ref{thm-sing-char}.
\end{enumerate}
\end{thm}

\begin{proof}[Proof of Theorems~\ref{thm-char-alt-intro} 
and~\ref{thm-char-alt}]
These results follow directly from
Theorems~\ref{thm-sing-char} and~\ref{thm-sing-char>0}.
\end{proof}

We can now also characterize whether $\Cp(\{x_0\})$ is
zero or not in terms of various properties of
singular functions as follows.

\begin{thm} \label{thm-char-Cp=0}
Assume that 
$u$ is a singular function in $\Om$ with singularity at 
$x_0$, and extend
$u$ by letting $u=0$ on $X \setm \Om$.
Then the following are equivalent\/\textup{:}
\begin{enumerate}
\item \label{k-1}
$\Cp(\{x_0\})>0$\textup{;}
\item \label{k-2}
$u(x_0) < \infty$\textup{;}
\item \label{k-bdd}
$u$ is bounded\textup{;}
\item \label{k-3}
$u \in \Np(B_r)$ for some  $r>0$\textup{;}
\item \label{k-4}
$u \in \Np(X)$\textup{;}
\item \label{k-5}
$u \in \Np_0(\Om)$\textup{;}
\item \label{k-6}
$\int_{B_r} g_u^p \, d\mu < \infty$ for some  $r>0$\textup{;}
\item \label{k-7}
$g_u \in L^p(X)$. 
\end{enumerate}
\end{thm}

\begin{proof}
Assume first that $\Cp(\{x_0\})=0$, i.e.\ \ref{k-1} fails.
In this case, \ref{k-2} fails by Lemma~\ref{lem-existence-conseq=0},
while \ref{k-3},  \ref{k-5} and~\ref{k-6} fail by
Proposition~\ref{prop-Np0-Cp=0}.
Hence also \ref{k-bdd}, \ref{k-4} and 
\ref{k-7} fail.

Assume now instead that $\Cp(\{x_0\})>0$, i.e.\ \ref{k-1} is true.
Then \ref{k-5} is true by Proposition~\ref{prop-sing-char>0-extra},
and thus \ref{k-3}--\ref{k-7} all hold.
Finally, \ref{k-2} and \ref{k-bdd} hold by Theorem~\ref{thm-uniq>0}.
\end{proof}

\section{Superlevel set estimates and Green functions}
\label{sect-Green}

\emph{Recall the standing assumptions from the beginning of 
Section~\ref{case-Cp(x0)=0}.}

\medskip

The following result about superlevel sets of superharmonic functions 
generalizes (and has been inspired by) 
Lemma~3.5 in Holopainen--Shan\-mu\-ga\-lin\-gam~\cite{HoSha}. 
This result holds even without assuming that $\Om$ is connected,
i.e.\ for nonempty open $\Om$ with $\Cp(X \setm \Om)>0$.
Recall that
\[
	\cpessinf_E u  := \sup \{ k\in\R: \Cp(\{ x\in E: u(x)<k \})=0 \}.
\]

\begin{lem}   \label{lem-level-est}
Let $E\subset\Om$ be relatively closed and let $u>0$ be a 
superharmonic function in $\Om$ which is \p-harmonic in $\Om\setm E$
and such that  $\min\{u,k\}\in\Np_0(\Om)$ for all $k>0$.
Then there is a constant $\Lambda>0$ such that
\begin{align*}
\cp(\Om^b,\Om^a)=\cp(\Om^b,\Om_a) &= \Lambda(b-a)^{1-p},
\quad \text{when }
   0 \le a <b\le\cpessinf_E u, \\
\cp(\Om_b,\Om_a) = \cp(\Om_b,\Om^a) 
&= \Lambda(b-a)^{1-p},
\quad \text{when }
   0 \le a <b<\cpessinf_E u,
\end{align*}
where $\Om^b=\{x\in \Om:u(x)\ge b\}$, $\Om_a=\{x\in \Om:u(x)>a\}$
and we interpret $\infty^{1-p}$ as $0$.
\end{lem}

The set $A=\{x:u(x)=\infty\}$ is a so-called polar set,
and thus $\Cp(A)=0$, by Proposition~2.2 in
Kinnunen--Shanmugalingam~\cite{KiSh06} (or Corollary~9.51 in \cite{BBbook}).
Hence $\cpessinf_E u=\infty$ if and only if $\Cp(E)=0$, 
which in turn happens if and only if 
$\cp(E,\Om)=0$, by Lemma~6.15 in \cite{BBbook}, i.e.\
if and only if the lsc-regularized capacitary potential
of $(E,\Om)$ is identically zero.
In this case it also follows from Lemma~\ref{lem-level-est}
that $u$ must be unbounded
as $\cp(\Om^b,\Om)=\Lambda b^{1-p}>0$ for all $b>0$.

Note that $\Lambda =b^{p-1}\cp(\Om^b,\Om)$ whenever $b$ satisfies the assumptions
above.
In particular if $b=1$ is allowed, then
$\Lambda =\cp(\Om^1,\Om)$.
Note also that even when $E=\{x_0\}$, it is not necessary
for $u$ in Lemma~\ref{lem-level-est} to be a singular function, see
the double-pole function in Example~\ref{ex-double-pole}.

\begin{proof}[Proof of Lemma~\ref{lem-level-est}]
Note that since $u$ is lower semicontinuous, $\Om_a$ is open.
If $\cpessinf_E u=0$, there is nothing to prove and we may let $\Lambda =1$.
(If $E=\emptyset$, we consider $\cpessinf_E u$ and $\inf_E u$ to be $\infty$, as usual.)
As $\Om^\infty$ is a polar set, we have $\Cp(\Om^\infty)=0$ 
and thus $\cp(\Om^\infty,\Om^a)= \cp(\Om^\infty,\Om_a)=0$,
i.e.\ the first formula holds when $b=\infty$.
We assume therefore that $b<\infty$ 
and $\cpessinf_E u>0$
in the rest of the proof.

Let $k=\cpessinf_E u$ if $\cpessinf_E u < \infty$, and $b < k <\infty$ otherwise.
Then $\Cp(E \setm \Om^k)=0$.
As $u$ is continuous 
in $\Om \setm E$,
we see that $\Om^k \cup E$ must be relatively closed.
By Lemma~\ref{lem-representation-cap-pot}, $u_k/k$
is the lsc-regularized capacitary potential of $(\Om^k \cup E,\Om)$,
and thus of $(\Om^k,\Om)$, since $\Cp(E \setm \Om^k)=0$.
Hence, by Theorem~\ref{thm-cap-pot},
\[ 
  \cp(\Om^b,\Om_a)
  = \biggl(\frac{b}{k}-\frac{a}{k}\biggr)^{1-p} \cp(\Om^k,\Om)
  = k^{p-1}(b-a)^{1-p}  \cp(\Om^k,\Om).
\] 
Upon letting $\Lambda =k^{p-1}  \cp(\Om^k,\Om)$,
this proves one identity in the statement of the lemma.
The other three identities then follow from Theorem~\ref{thm-cap-pot}.

To see that $\Lambda >0$, we note
that as $u>0$, there is some 
$b$ such that $0<b<\cpessinf_E u$ and 
$\Cp(\Om^b)>0$,
and hence $\Lambda =b^{p-1}\cp(\Om^b,\Om)>0$, by Lemma~6.15 in~\cite{BBbook}.
\end{proof}

In the proof above we used that $\Om^k \cup E$ is relatively closed.
Observe that it is not always true that $\Om^k$ itself is relatively closed,
as seen by the following example.

\begin{example}
In unweighted $\R^3$ with $p=2$, $x_0=0$ 
and $x_j=(2^{-j},0,0)$, $j=1,2,\ldots$\,,
let
\[
   u(x)=\sum_{j=1}^{\infty}  \frac{4^{-j}}{|x-x_j|},
   \quad 
   \Om=B(0,1)
   \quad \text{and} \quad
   E=\{x_0,x_1,\ldots\}.
\]
By linearity and e.g.\
Lemma~7.3 in Hei\-no\-nen--Kil\-pe\-l\"ai\-nen--Martio~\cite{HeKiMa},
$u$ is superharmonic in $\R^3$ 
and harmonic in $\R^3\setm E$.
As $u(x_j)=\infty$, $j=1,2,\ldots$\,, and
$u(0)=1$, it follows that $\Om^k$ is not relatively closed when  
$k>1$.
\end{example}

Recall Definition~\ref{deff-Green-intro} of Green functions.
We can now 
relate singular and Green functions in the following way.

\begin{thm} \label{thm-normalized-sing} 
Let $v$ be a singular function in $\Om$ with singularity at $x_0$,
and let 
  \[
  \alpha = \begin{cases}
   \cp(\{x \in \Om: v(x) \ge 1\},\Om)^{1/(1-p)}, & \text{if } \Cp(\{x_0\})=0, \\[2mm]
    \displaystyle \frac{1}{v(x_0)}\cp(\{x_0\},\Om)^{1/(1-p)}, & \text{if } \Cp(\{x_0\})>0.
    \end{cases}
    \]
Then $u:=\alp v$ is a Green function in $\Om$ with singularity at $x_0$.
Moreover, \eqref{eq-normalized-Green-intro} holds for $u$, and
$\alp$ is the unique number such that $u$ 
is a Green function.
\end{thm}

\begin{proof}
Let $u=\alp v$ and $\Om^b=\{x\in \Om:u(x)\ge b\}$ 
for $b \ge 0$.
Clearly, $u$ is a singular function in $\Om$ with singularity at $x_0$.

If $\Cp(\{x_0\})=0$, then by the definition of $u$ and $\alp$,
\[
   \cp(\Om^{\alp},\Om)=\cp(\{x \in \Om: v(x) \ge 1\},\Om)
   = \alp^{1-p},
\]
and thus  $\Lambda=1$ in Lemma~\ref{lem-level-est}. 

On the other hand,
if $\Cp(\{x_0\})>0$ then $u(x_0)<\infty$, 
by Theorem~\ref{thm-uniq>0}, and $u/u(x_0)$
is a capacitary potential in $\Om$ for $\{x_0\}$, as well as for $\Om^{u(x_0)}$,
by Lemma~\ref{lem-representation-cap-pot}.
Hence,
\[
   \cp(\Om^{u(x_0)},\Om)=\cp(\{x_0\},\Om)
   = (\alp v(x_0))^{1-p} = u(x_0)^{1-p},
\]
and so $\Lambda=1$ in Lemma~\ref{lem-level-est} also in this case.

Now, \eqref{eq-normalized-Green-intro-deff} and 
\eqref{eq-normalized-Green-intro} follow from Lemma~\ref{lem-level-est}.
Since \eqref{eq-normalized-Green-intro-deff} holds, $\alp$ must be unique.
\end{proof}

By Lemma~\ref{lem-level-est}, it is enough if 
the normalization
\eqref{eq-normalized-Green-intro-deff}
holds for one $b$,
and we may e.g.\ let $b=\min\{1,u(x_0)\}$.
Thus a singular function is a Green function if and only if
\begin{equation}\label{eq-green-normalization}
     \begin{cases}
             \cp(\Om^1,\Om)=1, & \text{if } u(x_0) \ge 1, \\
             \cp(\Om^{u(x_0)},\Om)=u(x_0)^{1-p}, & \text{if } u(x_0) < \infty.
     \end{cases}
\end{equation}
When $1 \le u(x_0)<\infty$ it is enough to require either condition.
It is always true that $\cp(\Om^{u(x_0)},\Om)=\cp(\{x_0\},\Om)$,
and thus if $u(x_0) < \infty$ 
we may equivalently require
that 
\begin{equation} \label{eq-u-cap-x0}
    u(x_0)=\cp(\{x_0\},\Om)^{1/(1-p)}.
\end{equation}
Note that it can happen that 
$\Om^{u(x_0)} \ne \{x_0\}$, e.g.\
when $X=[0,\infty)$, $\Om=[0,2)$ and $x_0=1$,
in which case $\Om^{u(x_0)}=[0,1]$.

\begin{remark} \label{rmk-classical-Green}
In weighted $\R^n$ with a \p-admissible weight $w$,
the classical Green function is defined as an (extended real-valued) continuous
weak solution $u$, 
with zero boundary values on $\bdy\Om$ (in Sobolev sense),
of the equation
\[ 
\Div(w|\nabla u|^{p-2}\nabla u) = -\delta_{x_0}
\quad \text{in } \Om,
\] 
i.e.\
\[
\int_\Om |\grad u|^{p-2} \grad u \cdot \grad\phi  \,d\mu = \phi(x_0) 
\quad \text{for all } \phi\in C_0^\infty(\Om).
\]
Here $\delta_{x_0}$ is the Dirac measure at $x_0$
and $d\mu=w\,dx$.
In particular, $u$ is \p-harmonic in $\Om\setm\{x_0\}$.

As $C_0^\infty(\Om)$ is dense in $W^{1,p}_0(\Om,\mu)$, we can test
the above integral identity with 
$\phi=\min\{u,1\}\in W^{1,p}_0(\Om,\mu)$. 
If $u(x_0)\ge 1$, this
yields
\[
1=\phi(x_0) = \int_{u<1} |\grad u|^{p} \,d\mu =  \cp (\Om^1,\Om).
\]
On the other hand, if $u(x_0)<1$, then $u/u(x_0)$ is a capacitary potential
of $(\Om^{u(x_0)},\Om)$, 
by Lemma~\ref{lem-representation-cap-pot}, and it follows that 
\[
u(x_0) = \phi(x_0) = \int_{u<u(x_0)} |\grad u|^{p} \,d\mu = u(x_0)^p \cp (\Om^{u(x_0)},\Om).
\]
Hence~\eqref{eq-green-normalization} holds in both cases, and we
conclude that
Definition~\ref{deff-Green-intro} is equivalent to the
classical definition of Green functions in weighted $\R^n$. 
In Section~\ref{sect-Cheeger} we show that the corresponding
equivalence holds also in the metric setting for 
Cheeger--Green functions defined via the differential structures
introduced by Cheeger~\cite{Cheeg}. 
\end{remark}

Using 
the superlevel set estimates in Lemma~\ref{lem-level-est} and 
the Harnack inequality in
Proposition~\ref{prop-punctured-ball}, we 
can now prove Theorem~\ref{thm-compare-bdy-ball}.

\begin{proof}[Proof of Theorem~\ref{thm-compare-bdy-ball}]
Let $r>0$ be such that $B_{50\lambda r}\subset\Om$ and define
\[
m = \min_{\bdry B_r} u \quad \text{and} \quad M = \max_{\bdry B_r} u,
\]
which exist and are finite as $u$ is \p-harmonic (and thus continuous) in
$\Om \setm \{x_0\}$.
The weak minimum principle for superharmonic functions
shows that $B_r \subset\Om^m$. 
Hence, by Proposition~\ref{prop-punctured-ball} and
\eqref{eq-normalized-Green-intro-deff},
\[
M^{1-p} 
\simeq  m^{1-p}
= \cp(\Om^m,\Om) 
\ge \cp(B_r,\Om).
\]
If $u(x_0)=M<\infty$,
then $\Cp(\{x_0\})>0$, 
and thus by \eqref{eq-u-cap-x0},
\[
M^{1-p}=u(x_0)^{1-p}=\cp(\{x_0\},\Om)\le \cp(B_r,\Om).
\]
On the other hand, if $u(x_0)>M$, then 
$u=H_{\Om \setm \itoverline{B}_r} u \le M$ in $\Om \setm \itoverline{B}_r$,
by the comparison principle \eqref{eq-H-comparison-principle},
and thus $\Om_M\subset B_r$.
As $u$ is a Green function, it follows 
from \eqref{eq-normalized-Green-intro-deff} that $\Lambda=1$ in 
Lemma~\ref{lem-level-est}, which thus gives
\[
     M^{1-p} = \cp(\Om_M,\Om) \le \cp(B_r,\Om).
\]
Hence, in either case, 
\[
    m \simeq M \simeq \cp(B_r,\Om)^{1/(1-p)}.
    \qedhere
\]
\end{proof}

\begin{remark} \label{rmk-DGM}
As mentioned in the introduction,
Theorem~\ref{thm-compare-bdy-ball}
was obtained in some specific cases on metric spaces
by Danielli--Garofalo--Marola~\cite[Theorems~3.1, 3.3 and~5.2]{DaGaMa}.
More precisely, they required that $1<p<\lqq$,
where $\lqq=\sup \lQ$ and
\[
  \lQ  =\biggl\{q>0 : \text{there is $C_q$ so that } 
        \frac{\mu(B_r)}{\mu(B_R)}  \le C_q \Bigl(\frac{r}{R}\Bigr)^q 
        \text{ for } 0 < r < R < \infty
        \biggr\}. 
\]
They however also implicitly assumed that $\lqq \in \lQ$,
see \cite[eq.\ (2.2)]{DaGaMa}, and that $X$ is LLC,
through their use (at the bottom of p.~354) of Lemma~5.3 in 
Bj\"orn--MacManus--Shanmugalingam~\cite{BMS}.
(Here the LLC condition is the same as
in \cite{BMS} or 
Holopainen--Shanmugalingam~\cite{HoSha}.)
As the constant $C_2$ in Theorem~3.1 in \cite{DaGaMa} depends
on $r$, they did not
obtain \eqref{eq-compare-bdy-ball}
when $p=\lqq \in \lQ$.
Note also that $\lqq$ is not the natural exponent
for determining when $\Cp(\{x_0\})>0$, see 
Bj\"orn--Bj\"orn--Lehrb\"ack~\cite[Proposition~1.3]{BBLeh1}
and Remark~\ref{rem-zero-cap}.
\end{remark}

When $\Om\subset\R^n$ (unweighted) is a bounded domain,
then two singular functions having singularity at $x_0\in\Omega$ are
multiples of each other. This follows from the results of
Serrin~\cite{Ser1965} and 
Kichenassamy--Veron~\cite{KichVeron}.
More precisely, if $1<p\le n$ and $u$ and $v$
 are such singular functions,
then by Theorem~3 in Serrin~\cite{Ser1965} 
there are positive constants $C_1$ and $C_2$
such that $-\Delta_p u = C_1 \delta_{x_0}$ and 
$-\Delta_p v = C_2 \delta_{x_0}$ in $\Omega$.
Hence there is $\lambda\in\R$ such that $-\Delta_p (\lambda u) = C_2 \delta_{x_0}$ in $\Omega$.
Since $\lambda u=v=0$ on $\partial\Omega$ and the solutions of such equations are 
unique by Theorem~2.1 in~\cite{KichVeron}, we conclude that $v=\lambda u$ in $\Omega$.
Theorem~2.1 in~\cite{KichVeron} is stated for regular $\Om$, but the
uniqueness part does not require any regularity of $\Om$,
since (2.8) therein follows directly from
the corresponding identity, obtained using
the Gauss--Green formula,  in a ball containing $\Om$. 
It is this use of the Gauss--Green formula which makes
the uniqueness argument
only applicable in unweighted $\R^n$.

In our generality we have not been able to 
prove such uniqueness, 
but we can 
show that two singular functions 
with the same singularity are always comparable in $\Om$.
This is based on part \ref{thm-intro-c} of Theorem~\ref{thm-main-intro},
which gives a stronger comparability result for Green functions. We collect here also
the proofs of the other parts of that theorem.

\begin{proof}[Proof of Theorem~\ref{thm-main-intro}] 
\ref{thm-intro-b}
This is a less refined form of Theorem~\ref{thm-normalized-sing}.
    
\ref{thm-intro-a}
If $\Cp(X \setm \Om)=0$, then Proposition~\ref{prop-bdd-cp=0} 
shows that there is no singular function.
On the other hand, if $\Cp(X \setm \Om)>0$ then
the existence of a singular function follows
from Theorems~\ref{thm-existence} and~\ref{thm-uniq>0}.
In view of \ref{thm-intro-b} this shows \ref{thm-intro-a}.

\ref{thm-intro-c}
Let $r>0$ be so small that $50 \la B_r  \subset \Om$.
By Theorem~\ref{thm-compare-bdy-ball},
$u \simeq v$ in $\itoverline{B}_r$.

Let $u=v=0$ on $X \setm \Om$.
As $u$ is \p-harmonic in $\Om \setm \{x_0\}$ and 
$u \in \Nploc(X \setm \{x_0\})$,
we see that, by definition, 
$\oHpind{\Om \setm \itoverline{B}_r} u=u$ in 
$\Om \setm \itoverline{B}_r$,  
and similarly for $v$.
By the comparison principle in \eqref{eq-H-comparison-principle},
\[
    u 
    =\oHpind{\Om \setm \itoverline{B}_r} u 
\simeq \oHpind{\Om \setm \itoverline{B}_r} v
    = v
  \quad \text{in } \Om \setm \itoverline{B}_r.
\]
The last part, for $\Cp(\{x_0\})>0$, follows
from \ref{thm-intro-b} and Theorem~\ref{thm-uniq>0}.
\end{proof}

The comparability result for singular functions,
but with comparison constants also depending on $u$ and $v$,
now follows from 
Theorem~\ref{thm-main-intro}\,\ref{thm-intro-b}--\ref{thm-intro-c}. 
When $\Cp(\{x_0\})>0$, 
Theorem~\ref{thm-uniq>0} 
allows us to say more, namely that
singular functions are unique up to a multiplicative factor.
However, regardless of the value of $\Cp(\{x_0\})$, we have
the following characterization of singular functions,
which is a more general version of Theorem~\ref{thm-sing-iff-comp-intro},
valid also when $\Cp(\{x_0\})>0$. 

\begin{thm}   \label{thm-sing-iff-comp}
Let $u$  be a singular function in $\Om$ with singularity at $x_0$,
and let $v\colon\Om \to (0,\infty]$ be \p-harmonic
in $\Om \setm \{x_0\}$.
If $\Cp(\{x_0\})>0$, we also assume that  
$v$ is superharmonic in $\Om$ or that $v(x_0)=\lim_{x\to x_0} v(x)$.

Then $v$ is a singular function in $\Om$ with singularity at $x_0$
if and only if $v \simeq u$,
with comparison constants depending on $u$ and $v$.
\end{thm}

\begin{proof}[Proof of Theorems~\ref{thm-sing-iff-comp-intro} 
and~\ref{thm-sing-iff-comp}]
If $v$ is a singular function, then $v \simeq u$ by 
Theorem~\ref{thm-main-intro}\,\ref{thm-intro-b}--\ref{thm-intro-c}. 

Conversely, if $v \simeq u$ then $v$ 
automatically satisfies \ref{b.2} in Theorem~\ref{thm-sing-char}
since $u$ does 
(by Proposition~\ref{prop-existence-conseq}).
Moreover, if $\Cp(\{x_0\})=0$ 
then $u(x_0)=\lim_{x\to x_0} u(x)= \infty$,
and thus also $v(x_0)=\lim_{x\to x_0} v(x)= \infty$, i.e.\ \ref{a.2}
in Theorem~\ref{thm-sing-char} holds.
If $\Cp(\{x_0\})>0$ then \ref{a.1} or
\ref{a.2} is true by assumption.
Hence $v$ is a singular function by Theorem~\ref{thm-char-alt}.
\end{proof}

\begin{remark}
The extra assumption in Theorem~\ref{thm-sing-iff-comp} 
when $\Cp(\{x_0\})>0$ cannot be omitted.
Indeed, if 
$X=\R$ (unweighted), $\Om=(-1,1)$ and $x_0=0$, then all the functions
\[
v(x)=\begin{cases}
        a(x+1), &-1 < x <0, \\ 
        1-x, &0 \le x <1,
     \end{cases}
\quad \text{with } a>0,
\]
are \p-harmonic in $\Om\setm\{x_0\}$ and comparable to each other, but 
only the one with $a=1$ is superharmonic and singular in $\Om$ 
with singularity at $x_0$.
\end{remark}

\section{\texorpdfstring{\p}{p}-harmonic functions with poles}
\label{sect-pole}

\emph{Assume in this section that $X$ is complete
and that $\mu$ is doubling and supports a \p-Poincar\'e inequality.
We also fix $x_0 \in X$ and write $B_r=B(x_0,r)$ for $r>0$.}

\medskip

We shall now
apply our results to general \p-harmonic functions
with poles.
Note that there is no relation between $G$ and $U$
in the theorem below.

\begin{thm} \label{thm-pharm-pole}
Let $G$ and $U$ be arbitrary open sets containing $x_0$, such that
$G$ is bounded and $\Cp(X \setm G)>0$.
Let $u$ and $v$ be \p-harmonic functions in $U \setm \{x_0\}$ such that
\begin{equation} \label{eq-pole}
 u(x_0):=   \lim_{x \to x_0} u(x) 
    = \infty
\quad \text{and} \quad 
 v(x_0):=   \lim_{x \to x_0} v(x) 
    = \infty.
\end{equation}
Then the following are true\/\textup{:}
\begin{enumerate}
\item \label{r-a}
  $\Cp(\{x_0\})=0$\textup{;}
\item \label{r-b}
there is a bounded domain $\Om \ni x_0$ and $a\ge 0$
such that $u-a$ is a singular function in $\Om$ with singularity 
at $x_0$\textup{;}
\item \label{r-c}
there is $r_0>0$ such that if $0<r<r_0$ and  $x\in\bdry B_r$, then 
\begin{equation}\label{eq-u-G}
 u(x) \simeq \cp(B_r,G)^{1/(1-p)},
\end{equation}
where the comparison constants depend on $G$ and $u$, but not on $r$\textup{;}
\item \label{r-d}
there is $r_0>0$ such that
\[ 
u \simeq v
\quad \text{in } B_{r_0},
\] 
where the comparison constants depend on $u$ and $v$. 
  \end{enumerate}
\end{thm}

Note that also the radius $r_0$, for which \ref{r-c} and \ref{r-d} 
hold, depends on $u$ (and $v$).
This is easily seen by considering the function 
$|x|^{(p-n)/(p-1)}-c$ in $\R^n$,
$p<n$, for various constants $c \ge 0$.

However, Theorem~\ref{thm-pharm-pole}\,\ref{r-c}--\ref{r-d}
shows that all \p-harmonic functions with
a given pole (i.e.\ such that \eqref{eq-pole} holds)
have growth of the same order
near the pole.
For elliptic quasilinear equations \eqref{eq-dvgA=B}
on unweighted $\R^n$, this is
a classical result due to Serrin~\cite[Theorem~1]{Ser1965}.
On the contrary, results in Bj\"orn--Bj\"orn~\cite{BBpower} 
show that the so-called \emph{quasiminimizers}
(rather than minimizers)
of the \p-energy integral $\int g_u^p\,d\mu$
can have singularities of arbitrary order, depending 
on the quasiminimizing constant.
Quasiminimizers were introduced by 
Giaquinta and Giusti~\cite{GG1},~\cite{GG2} as a natural unification
of elliptic equations with various ellipticity constants.

\begin{proof}
Let $r'>0$ be such that $B_{r'} \Subset U$ and $\Cp(X \setm B_{r'})>0$,
and 
let
\[
M(r) = \max_{\bdry B_r} u
\quad \text{for } 0<r\le r'.
\]
Let $a=\max\{M(r'),0\}$, 
$\Om=\{x \in B_{r'} : u(x) > a\}$ and $\ub=u-a$.
By the strong maximum principle for \p-harmonic functions, $\Om$
must be connected.
It is easy to see that $\ub$ satisfies \ref{a.2} and \ref{b.1} in 
Theorem~\ref{thm-sing-char}
(with $\ub$ in place of $u$). 
As $\ub$ is \p-harmonic in $\Om \setm \{x_0\}$, it follows from  
Theorem~\ref{thm-char-alt}
that $\ub$ is a singular function in $\Om$, i.e.\ \ref{r-b} holds.
Thus \ref{r-a} follows from Theorem~\ref{thm-char-Cp=0}.

Let next 
$r_0>0$ be so small that $B_{50 \la r_0} \subset \Om$.
By the strong minimum principle for superharmonic functions, 
$\inf_{B_{r_0}}u >a \ge 0$ and thus
\[
u \ge u-a \ge Cu \quad \text{in } B_{r_0},
\]
with $C>0$ depending on $a$ and $\inf_{B_{r_0}}u$. 
Theorems~\ref{thm-main-intro}\,\ref{thm-intro-b} 
and~\ref{thm-compare-bdy-ball}, applied to $\ub$, then yield
\begin{equation} \label{eq-u-Om}
    u(x) \simeq \ub(x) \simeq \cp(B_r,\Om)^{1/(1-p)}
\quad \text{whenever } x\in\bdry B_r \text{ and } 0<r<r_0,
\end{equation}
where the comparison constants depend on $u$, $a$ and $r_0$.
This proves \ref{r-c} for
$G=\Om$. 
Also \ref{r-d} follows directly from this, with the same
choice of $r_0$.

Now consider a general open set $G$ in \ref{r-c}.
We may choose $r'$ above so small that $B_{r'} \subset G$.
It follows that $\Om \subset G$.
For $0 < r \le r_0$, let
$u_r$ be the capacitary potential
for $B_r$ in $G$, and set $a_r = \max_{\bdy \Om} u_r$.
Then $0 < a_r \le a_{r_0}< 1$. 
Also let
\[
   v_r=\frac{u_r-a_r}{1-a_r}.
\]
Then $v_r=1$ in $B_r$ and $v_r\le0$ on $X \setm \Om$.
Hence
\begin{align*}
  \cp(B_r,\Om) 
&\le \int_X g_{v_r}^p \, d\mu
\le \biggl(\frac{1}{1-a_r}\biggr)^p \int_X g_{u_r}^p \, d\mu \\
&\le \biggl(\frac{1}{1-a_{r_0}}\biggr)^p \int_X g_{u_r}^p \, d\mu
 = \biggl(\frac{1}{1-a_{r_0}}\biggr)^p  \cp(B_r,G).
\end{align*}
As $\cp(B_r,G) \le \cp(B_r,\Om)$, we
see that  \eqref{eq-u-G} follows from \eqref{eq-u-Om}.
\end{proof}

\section{Local assumptions}
\label{sect-local-assumptions}

In this section we investigate to which extent
our results hold in more general metric measure spaces than those assuming
our three standing assumptions: completeness, doubling measure
and \p-Poincar\'e inequality.
We start by introducing the local assumptions.

\begin{deff}
The measure \emph{$\mu$ is doubling within a ball $B_0$}
if there is $C>0$ (depending on $B_0$)
such that 
\[
\mu(2B)\le C \mu(B) \quad \text{holds for all balls } B \subset B_0.
\]
Similarly, 
the \emph{\p-Poincar\'e inequality holds within a ball $B_0$} 
if there are constants $C>0$ and $\lambda \ge 1$ (depending on $B_0$)
such that \eqref{eq-deff-PI} holds for all balls $B\subset B_0$, 
all integrable functions $u$ on $\la B$, and all 
\p-weak upper gradients $g$ of $u$
within $\la B$.

We also say that any of the above two properties is \emph{local}
if for every $x_0 \in X$ there is $r_0$ (depending on $x_0$) 
such that the property holds within $B(x_0,r_0)$.
If a property holds within every ball $B(x_0,r_0)$ then it is called 
\emph{semilocal}.
\end{deff}

Note that if $\mu$ is semilocally doubling and
$C$ is independent of $x_0$ and $r_0$,
then $\mu$ is doubling according to \eqref{eq:doubling-x0}.
The situation is similar for Poincar\'e inequalities.

The following result from Bj\"orn--Bj\"orn~\cite{BBsemilocal} makes
it possible to generalize the results in this paper to spaces
with local assumptions.
Recall that a space is \emph{proper} 
if every bounded closed subset is compact.

\begin{thm} \label{thm-BBsemilocal}
\textup{(Proposition~1.2 and Theorem~1.3 in 
\cite{BBsemilocal})}
If $X$ is proper and connected, and $\mu$ is locally doubling
and supports a local \p-Poincar\'e inequality, 
then $\mu$ is semilocally doubling
and supports a semilocal \p-Poincar\'e inequality.
\end{thm}

Examples in~\cite{BBsemilocal}
show that properness cannot be replaced by completeness, and
connectedness cannot be dropped from Theorem~\ref{thm-BBsemilocal}.
Moreover, if $\mu$ supports a semilocal Poincar\'e inequality, 
then $X$ is connected.

\medskip

\emph{So, for the rest of this section, we assume that
$X$ is proper and connected, and that $\mu$ is locally doubling
and supports a local \p-Poincar\'e inequality.}

\medskip

As in Keith--Zhong~\cite[Theorem~1.0.1]{KeZh},
a better semilocal $q$-Poincar\'e inequality with some $q<p$ holds
also in this case, by Theorem~5.3 in~\cite{BBsemilocal}.

In \cite[Section~10]{BBsemilocal}, it was explained
how the potential theory of \p-harmonic functions, 
specifically the results in Chapters~7--14 in \cite{BBbook},
hold under these assumptions, with the exception of
the Liouville theorem.
The same is true for the results in this paper,
it is only the dependence of constants on the different
associated parameters that needs to be carefully investigated.
If $X$ is bounded, then the semilocal assumptions are global and 
hence our standing assumptions are equivalent to
the local assumptions above in this case.

If $X$ is unbounded, we let (as before) $\Om$ be a bounded domain
and find a ball $B_0\supset \Om$.
Since $X$ is unbounded, the condition $\Cp(X\setm \Om)>0$ is automatically
satisfied.
Let $\CPI$, $\la$ and $C_\mu$ be the constants
in the \p-Poincar\'e inequality and the doubling condition
within $2 B_0$.
The weak Harnack inequalities then hold for every ball 
$B$ such that $50 \la B \subset\Om$ and with a constant
depending only on $p$, $\CPI$, $\la$ and $C_\mu$,
coming from $2B_0$ as above.
Thus all our estimates depend on these parameters instead
of the constants in the global assumptions, which perhaps do not hold
on~$X$.

\section{Holopainen--Shanmugalingam's definition}
\label{sect-HS}

In this section we compare
our results with the following
definition of singular functions from Holopainen--Shanmugalingam~\cite{HoSha}.
(See below for the precise assumptions on $X$.)

\begin{deff} \label{def-HS-green-Omega}
(Definition~3.1 in \cite{HoSha})
Let $\Omega\subset X$ be a relatively compact domain.
A function
$u\colon X\to [0,\infty]$ is a \emph{singular function 
in the sense of Holopainen--Shanmugalingam}, or
an \emph{HS-singular function}, 
in $\Omega$
with singularity at $x_0 \in \Om$ if

\addjustenumeratemargin{(HS1)}
\begin{enumerate}
\renewcommand{\theenumi}{\textup{(HS\arabic{enumi})}}%
 \item \label{HS-a}
 $u$ is \p-harmonic in $\Omega\setminus\{x_0\}$ and positive in $\Om$\textup{;}
 \item \label{HS-b}
$u|_{X\setminus\Omega} = 0$ q.e.\textup{;}
 \item \label{HS-b2}
$u\in\Np(X\setminus B_r)$
   for all $r>0$\textup{;}
 \item \label{HS-c}
 $\lim_{x\to x_0} u(x) = \cp(\{x_0\},\Omega)^{1/(1-p)}$
  (in particular, $\lim_{x\to x_0} u(x)=\infty$ if $\cp(\{x_0\},\Omega)=0$)\textup{;}
 \item \label{HS-d}
For $0\le a < b<\sup_\Om u$,
\begin{equation}\label{eq-explicit}
\biggl(\frac{p-1}{p}\biggr)^{2(p-1)}(b-a)^{1-p}
\le \cp(\Omega^b,\Omega_a)\le p^2 (b-a)^{1-p},
\end{equation}
 where $\Omega^b = \{x\in\Omega : u(x)\ge b\}$ and 
 $\Omega_a = \{x\in\Omega : u(x) > a\}$.
\end{enumerate}
\end{deff}

The existence of such a function, when $x_0\in\Omega\subset X$ and $\Omega$ is
a relatively compact domain, 
was given in~\cite[Theorem~3.4]{HoSha}
under the assumptions 
that $X$ is connected, locally compact, noncompact and satisfies the 
so-called LLC
property, and that $\mu$ is locally doubling and supports a local
$q$-Poincar\'e inequality for some $1\le q<p<\infty$,
cf.\ Remark~2.4 in \cite{HoSha}.
These local assumptions are as defined in \cite{HoSha} and
are stronger than those in Section~\ref{sect-local-assumptions}.
In fact, they coincide with those called semiuniformly local in 
Bj\"orn--Bj\"orn~\cite{BBsemilocal}.

\begin{remark} \label{rmk-HS-existence}
From the proof of \cite[Theorem~3.4]{HoSha} it is not clear 
why the function called $g$ on p.\ 322 therein satisfies \ref{HS-b2}
in the case when $\Cp(\{x_0\})=0$.
This can be justified, at least under the assumptions in this
paper, in a similar way as we do in Lemma~\ref{lem-bdd-u-imp-Np1},
using Perron solutions and the uniqueness result in
Theorem~\ref{thm-Perron-fund}\,\ref{P-C-uniq}.
These tools were however not available at that time.
\end{remark}

In the definition of HS-singular functions above,
the value $u(x_0)$ can be rather arbitrary.
In particular, $u$ is not required to be superharmonic in $\Om$.
However, in order for \ref{HS-d} to be satisfied, one must have
$0 <u(x_0) \le \cp(\{x_0\},\Omega)^{1/(1-p)}$ 
(which automatically holds if $\Cp(\{x_0\})=0$).
In view of \ref{HS-c} it is natural to
let $u(x_0):=\lim_{x \to x_0} u(x)$, and we do so from now on.

We obtain the following relation to our 
Definitions~\ref{deff-sing-intro} and~\ref{deff-Green-intro}.

\begin{prop}\label{prop-HS-vs-us}
Assume that $X$ is a proper connected
metric space,
and that $\mu$ is locally doubling
and supports a local \p-Poincar\'e inequality.
Let $\Om \subset X$ be a bounded domain 
such that $\Cp(X \setm \Om)>0$, and let $x_0 \in \Om$
and $u\colon X \to [0,\infty]$.
\begin{enumerate}
\item \label{i-hs-a}
If $u$ is an HS-singular function in $\Om$ with singularity at $x_0$,
and $u(x_0)=\lim_{x \to x_0} u(x)$, then
$u|_\Om$ is a singular function in $\Om$ in the sense of 
Definition~\ref{deff-sing-intro}.
\item \label{i-hs-b}
If $u$ is a Green function in $\Om$ with singularity at $x_0$
in the sense of Definition~\ref{deff-Green-intro},
then its zero extension $\ut$\/
\textup{(}given by letting $\ut=0$ on $X \setm \Om$\textup{)}
is an HS-singular function in $\Om$ with singularity at $x_0$.
\end{enumerate}

In particular there is an HS-singular function in $\Om$ with singularity at $x_0$.
\end{prop}

\begin{proof}
By the discussion in Section~\ref{sect-local-assumptions}
the results in this paper hold under these assumptions on $X$.
Part~\ref{i-hs-a} follows from Theorem~\ref{thm-char-alt-intro},
while part~\ref{i-hs-b} follows from the definition
of Green functions and Theorem~\ref{thm-normalized-sing}
(which yields \ref{HS-c}).
Finally, the existence follows from
Theorem~\ref{thm-main-intro}\,\ref{thm-intro-a}.
\end{proof}

Requiring the superlevel set estimates in \eqref{eq-explicit} with those
explicit constants, is a weaker type of normalization than
in our definition of Green functions.
However, it was natural in \cite{HoSha} as it used the best
estimates available at the time.

We also remark that while proving the 
existence of HS-singular functions,
Ho\-lo\-pai\-nen and 
Shan\-mu\-ga\-lin\-gam~\cite[formula~(8), p.\ 322]{HoSha} showed that 
estimate~\eqref{eq-compare-bdy-ball} holds,
for $x$ close enough to $x_0$,
for the HS-singular functions obtained by their construction.
Here, we obtain it for all Green functions.
Recall that in this generality, it is not known whether Green functions are
unique when $\Cp(\{x_0\})=0$.

\section{Cheeger--Green functions}
\label{sect-Cheeger}

\emph{Recall the standing assumptions  from
the beginning of Section~\ref{sect-harm}.
In this section we also assume that $\Om$ is bounded
and that $\Cp(X \setm \Om)>0$.}

\medskip

Theorem~4.38 in Cheeger~\cite{Cheeg} shows that,
under our standing assumptions,  the metric space $X$ can be
equipped with a coordinate structure in such a way that 
each Lipschitz function $u$ in $X$ has a vector-valued ``gradient'' 
$D u$, defined a.e.\ in $X$.
Since Lipschitz functions are dense in $\Np(X)$,
this gradient can be extended uniquely to
$\Np(X)$, by Franchi--Haj\l asz--Koskela~\cite[Theorem~10]{FHK} or Keith~\cite{Ke-meas}.
Then $|Du|\simeq g_u$ a.e.\ in $X$ for all $u\in \Np(X)$, 
where the comparison constants are independent of $u$. 
Here and throughout this section
$|\cdot|$ is an inner product norm on some $\R^N$,
related to the Cheeger structure.
Both $|\cdot|$ and the dimension $N$ depend on $x \in X$,
but there is a bound on $N$, which only depends on the doubling
constant and the constants in the Poincar\'e inequality.
By adding dummy coordinates, it can thus be assumed that $Du(x)\in\R^N$,
with the same dimension $N$ for all $x$.

In a general metric space $X$ there is some freedom in choosing the Cheeger
structure.
In (weighted) $\R^n$ we will however always make the natural choice
$Du = \nabla u$, where
$\nabla u$ denotes the Sobolev gradient from
Hei\-no\-nen--Kil\-pe\-l\"ai\-nen--Martio~\cite{HeKiMa}.
In this case $|Du|=g_u$, by Proposition~A.13
in~\cite{BBbook}.
If the weight $w$ on $\R^n$
satisfies $w^{1/(1-p)}\in L^1\loc(\R^n)$ (in particular, if it
is a Muckenhoupt $A_p$ weight) then the Sobolev gradient $\nabla u$
is also the distributional gradient.

It was shown by Haj\l asz and Koskela that
$g_u= |\nabla u|$ also on Riemannian 
manifolds~\cite[Proposition~10.1]{HaKo2000}
and Carnot--Carath\'eodory
spaces~\cite[Proposition~11.6 and Theorem~11.7]{HaKo2000}, equipped
with their natural measures.

\emph{Cheeger\/ \textup{(}super\/\textup{)}minimizers} and
\emph{Cheeger \p-harmonic functions} are defined by replacing $g_u$ and $g_{u+\phi}$ in
Definition~\ref{def-quasimin} with $|Du|$ and $|D(u+\phi)|$. 
Similarly, the \emph{Cheeger variational
capacity} of $E\subset \Om$, denoted $\cpp(E,\Om)$,
and the related capacitary potentials are defined as in
Section~\ref{sect-cap-pot} but with $g_u$ replaced by $|Du|$.
Then all the results we have obtained in the previous sections hold also for the 
corresponding \emph{Cheeger singular} and \emph{Cheeger--Green functions},
which are defined as in
Definitions~\ref{deff-sing-intro} and~\ref{deff-Green-intro},
with obvious modifications.
See Appendix~B.2 in~\cite{BBbook}
for more comments, details and references on 
Cheeger \p-harmonic functions
in general,
and Danielli--Garofalo--Marola~\cite[Section~6]{DaGaMa}
for some specific results 
for Cheeger singular and Cheeger--Green functions.

Due to the additional vector structure of the Cheeger gradient 
it is possible to
make the following definition, which has no counterpart in the case of
general scalar-valued upper gradients.

\begin{deff}   \label{def-supersol}
A function $u\in N^{1,p}\loc(\Omega)$ is a
\emph{\textup{(}super\/\textup{)}solution} in $\Omega$ if
\begin{equation} \label{eq-Cheeger-soln}
\int_{\Omega}|Du|^{p-2}Du\cdot D\varphi\,d\mu\ge 0
           \quad \text{for all (nonnegative) } \phi \in \Lipc(\Om),
\end{equation}
where $\cdot$ is the inner product giving rise to the 
norm $|\cdot|$,
and $\Lipc(\Om)$ denotes the family of Lipschitz functions
with compact support in $\Om$. 
\end{deff}

For solutions, one can  equivalently
replace $\ge$ by $=$ in \eqref{eq-Cheeger-soln},  
which follows directly after testing also with $-\phi$.

It can be shown that a function is a (super)solution if and only
if it is a Cheeger (super)minimizer, the proof is the same as for
Theorem~5.13 in
Hei\-no\-nen--Kil\-pe\-l\"ai\-nen--Martio~\cite{HeKiMa}.
In weighted $\R^n$ with a \p-admissible weight and
the choice $Du = \nabla u$, we have
$g_{u}=|Du|=|\nabla u|$ a.e.\ which implies that
(super)minimizers, Cheeger (super)minimizers and (super)solutions
coincide, and are the same as in \cite{HeKiMa}.  
Similar identities hold also on Riemannian
manifolds and Carnot--Carath\'eodory spaces
equipped with their natural measures.

The following result is contained in Proposition~5.1 
in Bj\"orn--Bj\"orn--Latvala~\cite{BBLat2},
see also Proposition~3.5 and Remark~3.6 in
Bj\"orn--MacManus--Shanmugalingam~\cite{BMS}.

\begin{prop}    \label{prop-supersol-Radon}
For every supersolution $u$
in $\Om$ there is a Radon measure $\nu \in \Np_0(\Om)'$
such that for all $\phi \in \Np_0(\Om)$,
\begin{equation}  \label{eq-deff-Tu}
\int_\Om|Du|^{p-2}Du\cdot D\phi\, d\mu = \int_\Om \phi\, d\nu,
\end{equation}
where $\cdot$ is the inner product giving rise to the 
norm $|\cdot|$. 
\end{prop}

Next we show that the Cheeger--Green functions are
exactly the weak solutions of the \p-Laplace equation 
with the Dirac measure 
on the right-hand side and with zero
boundary values, as in the case of  $\R^n$
considered in Remark~\ref{rmk-classical-Green}.

\begin{thm}\label{thm-cheeger-delta}
Let $u$ be a Cheeger--Green function in $\Omega$ with singularity at
$x_0$.
Then 
\begin{equation}\label{eq-cheeger-green-eq}
\int_\Om |D u|^{p-2} D u \cdot D \phi  \,d\mu = \phi(x_0) 
\quad \text{for all } \phi \in \Lipc(\Om),
\end{equation}
that is, $\Delta_p u = -\delta_{x_0}$ in the weak sense.

Conversely, assume that
$v$ is an\/ \textup{(}extended real-valued\/\textup{)} 
continuous function in $\Om$
such that $|D v|\in L^{p-1}(\Omega)$, 
\ref{dd-Np0} in Definition~\ref{deff-sing-intro} is satisfied,
and $v$ is a solution of~\eqref{eq-cheeger-green-eq}.
Then $v$ is a Cheeger--Green function.
\end{thm}

Note that the assumption $|D v|\in L^{p-1}(\Omega)$ 
in the second part of the statement is natural, since 
it guarantees that the integral 
in~\eqref{eq-cheeger-green-eq} is well-defined,
and it moreover holds for all superharmonic functions,
by
Theorem~5.6 in Kinnunen--Martio~\cite{KiMa}
(or \cite[Corollary~9.55]{BBbook}).

\begin{proof}
Assume first that $\Cp(\{x_0\})=0$.
Write $u_k=\min\{u,k\}$ for $k>0$. 
Then $u_k\in \Np_0(\Om)$ by 
Proposition~\ref{prop-existence-conseq}\,\ref{ee-Np2},
and $u_k$ is a supersolution. 
Let $\nu_k \in \Np_0(\Om)'$ be the corresponding Radon measures given 
by Proposition~\ref{prop-supersol-Radon}. 
Since $u_k$ is Cheeger \p-harmonic in $\Om \setm \Om^k$,
$\nu_k$ is supported on $\Om^k$.
Hence,
by testing~\eqref{eq-deff-Tu} for $\nu_k$ with
$\phi=u_k$, we obtain that
\begin{equation}    \label{eq-k-nu-k}
\int_\Om|Du_k|^{p}\, d\mu = \int_\Om u_k\, d\nu_k = k\nu_k(\Om^k).
\end{equation}
On the other hand, the function $u_k/k$ is the
Cheeger capacitary potential of $(\Omega^k,\Omega)$,
by Lemma~\ref{lem-representation-cap-pot}.
Thus it follows from the 
normalization~\eqref{eq-normalized-Green-intro-deff} of 
Cheeger--Green functions that
\begin{equation}    \label{eq-k-Ch-cap}
\int_\Om|Du_k|^{p}\, d\mu = k^p \cpp(\Om^k,\Om) = k^p k^{1-p} = k,
\end{equation}
and so $\nu_k(\Om^k)=1$ for all $k>0$. 

Let $\phi \in \Lipc(\Om)$ and let $\eps>0$. 
Choose $k_0>0$ so large that
$|\phi(x)-\phi(x_0)|<\eps$ for all $x\in\Omega^{k_0}$ 
(and hence also for all $x\in\Omega^{k}$ whenever $k\ge k_0$); 
note that this is possible by Theorem~\ref{thm-compare-bdy-ball}
and Proposition~\ref{prop-existence-conseq}\,\ref{ee-bdd}.
Then \eqref{eq-deff-Tu}
and the fact that $\nu_k(\Om)=\nu_k(\Om^k)=1$ yield
\begin{align*}
\biggl| \int_\Om |D u_k|^{p-2} D u_k \cdot D \phi  \,d\mu - \phi(x_0) \biggr|
& = \biggl|\int_\Om \phi\, d\nu_k - \phi(x_0) \biggr|\\
& \le \int_{\Om^k} |\phi-\phi(x_0)| \,d\nu_k \le \eps
\end{align*}
for all $k\ge k_0$. 
Since $|D u|\in L^{p-1}(\Omega)$ by 
Theorem~5.6 in Kinnunen--Martio~\cite{KiMa}
(or \cite[Corollary~9.55]{BBbook})
and $\phi \in \Lipc(\Om)$,
we see that
\[
\bigl\lvert|D u_k|^{p-2} D u_k \cdot D \phi \bigr\rvert
\le |D u_k|^{p-1} |D\phi| \le |D u|^{p-1} \|D\phi\|_\infty \in L^{1}(\Omega)
\]
for all $k >0$.
As $D u_k\to Du$ a.e.\ in $\Omega$, we hence obtain 
by dominated convergence that
\[
\biggl| \int_\Om |D u|^{p-2} D u \cdot D \phi  \,d\mu - \phi(x_0) \biggr| \le \eps.
\]
Since this holds for all $\eps>0$, the claimed 
identity~\eqref{eq-cheeger-green-eq} follows when $\Cp(\{x_0\})=0$.

Next, consider the case when $\Cp(\{x_0\})>0$. 
Then we know by Theorem~\ref{thm-uniq>0}
that $u \in \Np_0(\Om)$ and $u$ is Cheeger \p-harmonic in $\Om\setm\{x_0\}$.
Let $\nu$ be the measure provided for $u$ by 
Proposition~\ref{prop-supersol-Radon}.
Since $u$ is Cheeger \p-harmonic in $\Om\setm\{x_0\}$, $\nu$ must be 
supported on $\{x_0\}$ and hence $\int_\Om\phi\,d\nu = \phi(x_0)\nu(\{x_0\})$
for all $\phi\in\Np_0(\Om)$.
Testing \eqref{eq-deff-Tu} with $\phi=u$
then shows as in \eqref{eq-k-Ch-cap} and \eqref{eq-k-nu-k} that
\[
  u(x_0)^{1-p} = \cpp(\Om^{u(x_0)},\Om) 
  =  \frac{1}{u(x_0)^p}  \int_\Om |Du|^p \,d\mu
  = u(x_0)^{1-p} \nu(\{x_0\}),
\]
i.e.\ $\nu(\{x_0\})=1$, which proves the claim when $\Cp(\{x_0\})>0$.

Conversely, 
let $v$ be as in the statement of the theorem. 
Then it is immediate that $v$ is Cheeger \p-harmonic  in $\Omega\setm\{x_0\}$.
Hence $v$ is a Cheeger
singular function by Theorem~\ref{thm-char-alt} 
with \ref{a.2} and~\ref{b.1}.
The normalization~\eqref{eq-green-normalization} for $v$ is now
obtained exactly as in Remark~\ref{rmk-classical-Green}, with $\grad u$ 
replaced by $Dv$, and thus $v$
is a Cheeger--Green function.
\end{proof}

\end{document}